\documentclass[a4paper,10pt,english,reqno]{amsart} 

\DeclareSymbolFontAlphabet{\mathcal}{symbols}

\usepackage{multirow}
\usepackage{multicol}

\usepackage{soul} % para tachar palabras
\usepackage[T1]{fontenc}
\usepackage{tikz}
\usetikzlibrary{fadings}
\usepackage[a4paper]{geometry}
\usepackage{pgfplots}
\pgfplotsset{compat=1.13}
\usepackage{mathrsfs}
\usetikzlibrary{arrows}
\usepackage{mathtools}
\usepackage{framed}
\usepackage{diagbox,multirow}
\usepackage{color}
\usepackage{amstext}
\usepackage{amsthm}
\usepackage{amssymb}
\usepackage{varioref}
\usepackage{amssymb}
\usepackage{enumerate}
\usepackage[cmtip,all]{xy}
\usepackage{pifont}
\usepackage{graphicx}
\usepackage{url}
\usepackage{float}
\usepackage[shortlabels]{enumitem}
\newcounter{saveenum}
\usepackage[pagebackref=true]{hyperref}
\usepackage{hyperref}
\usepackage{backref}
\hypersetup{
    colorlinks,
    citecolor=blue,
    filecolor=black,
    linkcolor=blue,
    urlcolor=black
}

\makeatletter
\@namedef{subjclassname@2020}{%
  \textup{2020} Mathematics Subject Classification}
\makeatother
\title[A reasonable notion of dimension]{A reasonable notion of dimension for singular intersection homology}

\date{\today}

\author{David Chataur}
\address{Lamfa\\
Universit\'e de Picardie Jules Verne\\
33, rue Saint-Leu\\
80039 Amiens Cedex~1\\
         France}
\email{David.Chataur@u-picardie.fr}

\author{Martintxo Saralegi-Aranguren}
\address{Laboratoire de Math{\'e}matiques de Lens\\  
      EA 2462 \\
      Universit\'e d'Artois\\
         SP18, rue Jean Souvraz\\
          62307 Lens Cedex\\
         France}
\email{martin.saraleguiaranguren@univ-artois.fr}

\author{Daniel Tanr\'e}
\address{D\'epartement de Math{\'e}matiques\\
         UMR-CNRS 8524 \\
         Universit\'e de Lille\\
         59655 Villeneuve d'Ascq Cedex\\
         France}
\email{Daniel.Tanre@univ-lille.fr}

\thanks{The first author was supported by the research project ANR-18-CE93-0002  ``OCHOTO'' . 
The third author was partially supported by 
the Proyecto PID2020-114474GB-100
and 
the ANR-11-LABX-0007-01  ``CEMPI''}



\makeatletter
\renewcommand\l@subsection{\@tocline{2}{0pt}{2pc}{5pc}{}}
\renewcommand\l@subsubsection{\@tocline{3}{0pt}{4pc}{10pc}{}}
\makeatother
\theoremstyle{plain}

\newtheorem{theoremv}{Main Theorem}

\newtheorem{proposition}{Proposition}[section]
\newtheorem{theoremb}[proposition]{Theorem}
\newtheorem{lemma}[proposition]{Lemma}
\newtheorem{corollary}[proposition]{Corollary}

\theoremstyle{definition}
\newtheorem{definition}[proposition]{Definition}
\newtheorem{example}[proposition]{Example}

\newtheorem{remark}[proposition]{Remark}

\numberwithin{equation}{section}
\newcommand{\secref}[1]{Section~\ref{#1}}
\newcommand{\subsecref}[1]{Subsection~\ref{#1}}
\newcommand{\thmref}[1]{Theorem~\ref{#1}}
\newcommand{\propref}[1]{Proposition~\ref{#1}}
\newcommand{\lemref}[1]{Lemma~\ref{#1}}
\newcommand{\corref}[1]{Corollary~\ref{#1}}

\newcommand{\exemref}[1]{Example~\ref{#1}}

\newcommand{\defref}[1]{Definition~\ref{#1}}

\def\R{{\mathbb R}}

\def\ov{\overline}

\def\cAb{{\mathcal Ab}}
\def\cB{{\mathcal B}}
\def\cC{{\mathcal C}}

\def\cF{{\mathcal F}}

\def\cP{{\mathcal P}}
\def\cT{{\mathcal T}}

\def\cS{{\mathcal S}}
\def\cU{{\mathcal U}}

\def\crP{{\mathscr P}}

\def\1{{\mathbf 1}}
\def\tc{{\mathtt c}}

\def\tv{{\mathtt v}}
\def\tw{{\mathtt w}}
\def\N{\mathbb{N}}

\def\R{\mathbb{R}}
\def\Z{\mathbb{Z}}

\def\im{{\rm Im\,}}
\def\rank{{\rm rank}\,}

\def\id{{\rm id}}

\def\codim{{\rm codim\,}}

\def\sing{{\rm Sing}}
\def\singf{{\rm Sing}^{\cF}}

\def\pr{{\rm pr}}
\def\sd{{\rm sd}}
\def\rc{{\mathring{\tc}}}

\def\menos{\backslash}

\newcounter{ejemplo}

\newcounter{figura}

\def\diam{{\rm diam}}
\def\dist{{\rm dist}}
\def\eps{\varepsilon}
 \newcommand{\inte}[1]{\mathring{#1}} 
\renewcommand\1{\hbox{\ding{192}}}

\begin{document} 
\begin{abstract} 
M. Goresky and R. MacPherson intersection homology is also  defined from the singular chain complex of a
filtered space by H. King, with a key formula to make selections among singular simplexes.
This formula needs a notion of dimension for subspaces $S$ of an Euclidean simplex, which is usually taken as the 
smallest dimension of the skeleta containing~$S$.
Later, P. Gajer employed another dimension based on the dimension of polyhedra containing $S$. 
This last one allows traces of pullbacks of singular strata in the interior of the domain of a singular simplex.

In this work, we prove that the two corresponding intersection homologies are isomorphic
for Siebenmann's CS sets. In terms of King's paper, 
this means that  polyhedral dimension is a ``reasonable''  dimension.
The proof uses a Mayer-Vietoris argument which needs an adapted subdivision. 
With the polyhedral dimension, that is a subtle issue. 
General position arguments are not sufficient and we introduce  strong
general position. 
With it,  a stability is added  to the generic character and we can do an inductive cutting of each 
singular simplex. 
This decomposition is realised with pseudo-barycentric subdivisions 
where the new vertices are not barycentres but close points of them.
\end{abstract}

%%%%%%%%%%%%%%%%%%%

\maketitle
%%%%%%%%%%%%%%%%%%%
\tableofcontents

%%%%%%%%%%%
%

%
%%%%%%%%%%%%%%%%%%%
\section*{Introduction}

Intersection homology was introduced by Goresky and MacPherson to restore Poincar\'e duality
for some singular spaces called  pseudomanifolds.
They first defined it for PL-pseudomanifolds in \cite{GM1} and extended it to
topological pseudomanifolds in \cite{GM2} by using a derived category of complexes of sheaves. 
A presentation of intersection homology directly from the singular chain complex of a filtered space and its 
concordance with the initial definition is made by H. King in \cite{MR800845} for CS sets (\defref{def:csset}).
The starting point is a selection of some specific singular simplexes, called $\ov{p}$-allowable, 
from a perversity introduced as a sequence of integers 
$\ov{p}(n)$ (cf. \defref{def:perversite})
or as a map $\ov{p}\colon \cS_{X}\to {\Z}$, defined on the set of strata of $X$ and
taking the value 0 on the regular strata (cf. \defref{def:perversitegen}).
Before stating it formally in \defref{def:homotopygeom},  
let's contemplate the key formula for a stratum $S$ and a simplex $\sigma\colon \Delta^\ell\to X$,
\begin{equation}\label{equa:keyfirst}
\dim \sigma^{-1}S\leq \dim \Delta-\codim S+\ov{p}(S).
\end{equation}
For regular strata,  this inequality  reduces to
$\dim \sigma^{-1}S\leq \dim \Delta$ which is always true for any notion of dimension compatible with 
the inclusions of sets. So we can restrict ourselves to the singular strata.
Notice that \eqref{equa:keyfirst} contains  three different kinds of dimension:
\begin{itemize}
\item $\dim \Delta$ refers to the classical dimension of an Euclidean simplex;
\item $\codim S$ comes from the filtered dimension of $X$ as  in \defref{def:espacefiltre};
\item $\dim \sigma^{-1}S$ has to be specified.
\end{itemize}

\smallskip
Originally in \cite{MR800845}, King replaces  \eqref{equa:keyfirst} by ``$\sigma^{-1}S$ is included in the 
($\dim \Delta-\codim S+\ov{p}(S)$)-skeleton of $\Delta$''; i.e., the  dimension of $\sigma^{-1}S$ comes from the skeleta of $\Delta$.
This definition works perfectly and this is a common choice for many works in intersection homology, 
see the book of G. Friedman (\cite{FriedmanBook}) for an extensive repertoire of its properties.
It is easy to notice that, for a simplex $\Delta$ of dimension $\ell\geq 2$ and a perversity $\ov{p}\leq \ov{t}$,
the set $\sigma^{-1}S$ is included in the $(\ell-2)$-skeleton of $\Delta$. 
In  \cite{CST1,CST3}, another chain complex is built from singular simplexes $\sigma\colon \Delta\to X$, with a filtration
$\Delta=\Delta_{0}\ast\dots\ast\Delta_{n}$ verifying $\sigma^{-1}X_{i}=\Delta_{0}\ast\dots\ast\Delta_{i}$.
They form a simplicial set denoted by $\singf X$ that is not a Kan complex in general and allow a blown-up cohomology
(\cite{CST4}) which gives  a Poincar\'e duality (\cite{CST2}) with cap products.
For any of these two choices, all the information on the singular
strata are contained in the boundary of $\Delta$. 
This is not suitable if we are looking for a notion of intersection homotopy groups.
To achieve a presentation of intersection homology which fits with homotopy groups,
Gajer uses in \cite{MR1404919} a notion of dimension 
which allows traces of the pull-back of the singular strata in the interior of the singular simplexes.
For that, one says that a subspace  of an Euclidean simplex,
\emph{$A\subset \Delta$, is of  {dimension} less than or equal to $\ell$ 
if $A$ is included in a polyhedron $Q$ with $\dim Q\leq \ell$.}
Call it the \emph{polyhedral dimension.}
 
 \smallskip
Here, a $\ov{p}$-allowable simplex is a singular simplex verifying  \eqref{equa:keyfirst} with the polyhedral
dimension for $\sigma^{-1}S$. A $\ov{p}$-allowable chain is a linear combination of 
$\ov{p}$-allowable simplexes and the $\ov{p}$-intersection chain complex is formed of 
chains $\xi$ such that $\xi$ and its boundary $\partial \xi$ are $\ov{p}$-allowable chains.
In this work, we prove  that the homology of this complex coincides with the initial intersection homology,
for Siebenmann's CS sets (\defref{def:csset}). 
To do this, we use a method initiated by King (see \thmref{thm:gregtransformationnaturelle}) who gives
a series of properties which guarantee such identification. 
On this problem, King wrote
\emph{if one changes the above definition to any reasonable notion of dimension, 
one ends up with the same intersection homology groups.}
Taking this sentence as a definition for ``reasonable'' and the  construction of King as a reference, we can state the main result of this work (\thmref{thm:intersecgajer}) as follows.

\begin{theoremv}
The notion of polyhedral dimension is reasonable.
\end{theoremv}

In the proof of the Main Theorem, a ``sensitive issue'' is the existence of subdivisions 
for having a Mayer-Vietoris sequence.
In classical singular homology, they are the barycentric subdivisions which induce a chain map homotopic to the identity.
In intersection homology, we have to prove the existence of subdivisions keeping the $\ov{p}$-allowability property.
For our present situation, this is a subtle issue.
To illustrate it, consider a filtered space, $X$, of singular set $\Sigma$, and 
a singular simplex $\sigma \colon \Delta^2 \to X$ such that $\sigma^{-1}\Sigma = b_{\sigma}$, 
where $b_{\sigma}$ is the  barycentre of $\Delta^2$. 
The simplex $\sigma$ is $\ov t$-allowable  for the top perversity $\ov{t}$. 
But in the barycentric subdivision, the edges containing $b_{\sigma}$ are no more of $\ov t$-intersection.
To avoid this phenomenon, we act in an inductive way, as in the barycentric subdivision, but 
now the new vertices are not necessarily the barycentres but close points  to them. 
How to choose these points such that any element of this pseudo-barycentric subdivision 
is  a $\ov{p}$-allowable simplex?
At each step,  we have two simplexes 
 and we  slightly move one of them to ensure that the new simplexes are in general position. 
We can realize this operation if our process is  \emph{generic}. 
But, at the same time, we need that the simplexes of the previous steps remain in general position. This is a \emph{stability} condition.
The classical notion of general position does not meet these requirements (\exemref{exam:passtable})
and 
we introduce the \emph{notion of strong general position} in \defref{def:strongGP}.
Such subdivisions are called \emph{pseudo-barycentric subdivisions.}
In Propositions~\ref{prop:ExisB}, \ref{prop:sd} and \ref{prop:sdT}, 
we show that pseudo-barycentric subdivisions  preserving the $\ov{p}$-allowability exist, that they induce a chain
map homotopic to the identity and give a small simplexes $\ov{p}$-intersection theorem.

\smallskip
The preservation of $\ov{p}$-allowability  is a significative step but we also have to take care of $\ov{p}$-intersection
chains. For that, we prove that
the allowability default of the boundary of an allowable simplex is characterized by the existence of one particular
face, that we call the \emph{bad face.} 
The final stage for the existence of a Mayer-Vietoris sequence is a characterisation
of these bad faces in the case of a pseudo-barycentric subdivision, cf \propref{bad-1}. 
The architecture of the rest of the proof comes from King's original theorem (\cite[Theorem 10]{MR800845}).
The properties are now in place and the verification of the hypotheses of this theorem follows a classic pattern.

\smallskip
This is time to warn the reader on a possible confusion: 
the main part of \cite{MR1404919} is devoted to the definition and investigation
of a simplicial set $I_{\ov{p}}X$ associated to a filtered space and a perversity. It is known (\cite{MR1489215})
that the homology of $I_{\ov{p}}X$ is not the $\ov{p}$-intersection homology of~$X$.
The simplicial set $I_{\ov{p}}X$ does not appear in this work. 
We study it in  \cite{CST9} with a particular attention to its homotopy groups, 
the possibility of their topological invariance  
in terms of the  space $X$ and a Hurewicz theorem between them and $\ov{p}$-intersection homology.
 
% 
%%%%%%%%%%%%%%%%%%%%%%%%%%%%%%

\smallskip
\paragraph{\bf Outline of the paper}
\secref{sec:rappel} is a recall of the notions of filtered spaces, CS sets and perversities.
Intersection homology defined from the polyhedral dimension is presented in \secref{sec:lesdeux}. 
This section also contains
the computation for the cone and the compatibility with the relation of homotopy.
Strong general position is studied in \secref{sec:strong} and applied to the existence of pseudo-barycentric subdivisions
in \secref{sec:EPBS}.
The existence of Mayer-Vietoris exact sequences is established in \secref{sec:MV}.
Finally, the Main Theorem is proven in \secref{sec:twointersec}.

\medskip
\paragraph{\bf Notation.}
In the text, the letter $R$ denotes a Dedekind ring and $G$ an $R$-module. 
The singular chain complex of a topological space, $X$,
is denoted by $C_{*}(X;G)$, or $C_{*}(X)$ if there is no ambiguity.
The domain of a simplex $\sigma\in \sing\,X$ is denoted by
 $\Delta^\ell$, where $\ell$ is its dimension,
or  $\Delta_{\sigma}$, or simply $\Delta$,
depending on the parameter concerned at the place where it is used.
We denote by $\inte\Delta=\Delta\menos \partial \Delta$  the interior of $\Delta$.
The notation $V\triangleleft W$ means that $V$ is a face of a polyhedron $W$.
The word ``space'' means topological space.

%

%%%%%%%%%%%%%%
%%%%%%%%%%%%%%%
\section{Some reminders}\label{sec:rappel}

\begin{definition}\label{def:espacefiltre} 
A \emph{filtered space} is a  Hausdorff space,
$X$, endowed with a filtration by closed subspaces,
$$X_0\subseteq X_1\subseteq\ldots X_{n-1}\subsetneq X_n=X,$$
The
\emph{dimension} of $X$ is denoted by $\dim X=n$. 
The connected components, $S$, of $X_{i}\backslash X_{i-1}$ are the \emph{strata} of  $X$ 
and we write $\dim S=i$ and $\codim S=\dim X-\dim S$. 
(These dimensions can be formal and not necessarily related to a notion of geometrical dimension.)

The strata of  $X_n\backslash X_{n-1}$ are \emph{regular strata}; the other ones are \emph{singular strata}.  
The family of non-empty strata is denoted by $\cS_{X}$ (or $\cS$ if there is no ambiguity). 
The subspace $\Sigma _{X}=X_{n-1}$ is  the \emph{singular set,} sometimes also denoted by $\Sigma$.
Its complementary subset $X\menos \Sigma_{X}$ is called the \emph{regular subspace.}
The filtered space is said of \emph{locally finite stratification} if 
every point has a neighborhood that intersects only a finite number of strata.
\end{definition} 

 An open subset $U \subset X$ is a filtered space for the \emph{induced filtration} given by
$U_i = U \cap X_{i}$.
The product $M\times X$ with  a topological space is a filtered space for the \emph{product filtration} defined by
$\left(M \times X\right) _i = M \times X_{i}$. 
 If $X$ is compact, the open cone 
$\rc X = X \times [0,1[ \big/ X \times \{ 0 \}$ is endowed with the 
\emph{conical filtration} defined by
$\left(\rc X\right) _i =\rc X_{i-1}$,  $0\leq i\leq n+1$. 
By convention, 
$\rc \,\emptyset=\{ \tv \}$, where $\tv=[-,0]$ is the apex of the cone.

\begin{definition}\label{def:applistratifiee}
A \emph{stratified map,} $f\colon X\to Y$, is a continuous map between filtered spaces such that, for each stratum  $S\in\cS_{X}$,
there exists a unique stratum $S^f\in \cS_Y$ with $f(S)\subset S^f$ and $\codim S^f\leq \codim S$.
\end{definition}

A continuous map $f\colon X\to Y$ is  stratified  if, and only if, the pull-back of a stratum 
  $S' \in \cS_Y$ is empty or a union of strata of $X$,
$f^{-1}(S')=\sqcup_{i\in I}S_{i}$, with $\codim S'\leq \codim S_{i}$ for each $i\in I$.
Thus, a stratified map sends a regular stratum in a regular one but  the image of a singular stratum 
can be included in a regular one.

Let $X$ be a filtered space.
The  canonical injection of an open subset  $U\hookrightarrow X$,
the canonical projection,
$\pr\colon M\times X\to X$, 
 the maps
 $\iota_{m}\colon X\to M\times X$ with $x\mapsto (m,x)$, $m\in M$,
$\iota_{t}\colon X\to\rc X$ with $x\mapsto [x,t]$, $t\neq 0$,
are stratified for the filtered structures described above. 
In the following definition,
the product $X\times [0,1]$ is endowed with the product filtration.

\begin{definition}\label{def:homotopie}
Two stratified maps 
$f,\,g\colon X\to Y$
are \emph{homotopic} if there exists a stratified map,  
$\varphi\colon X\times [0,1]\to Y$,
such that  $\varphi(-,0)=f$ and $\varphi(-,1)=g$.
Homotopy is an equivalence relation and produces the notion of homotopy equivalence between filtered spaces.
\end{definition}

We present now a definition of the CS sets of  Siebenman, \cite{MR0319207}, in which the links of singular points
are not necessarily CS sets but only supposed to be non-empty filtered spaces. 
With this definition, the regular subspace $X\menos \Sigma$   is a dense open subset of $X$
and the stratification is locally finite.

\begin{definition}\label{def:csset}
A \emph{CS set} of dimension $n$ is a filtered space,
$$
\emptyset\subset X_0 \subseteq X_1 \subseteq \cdots \subseteq X_{n-2} \subseteq X_{n-1} \subsetneqq X_n =X,
$$
such that, for each $i$, 
$X_i\backslash X_{i-1}$ is a topological manifold of dimension $i$ or the empty set. 
Moreover, for each point $x \in X_i \backslash X_{i-1}$, $i\neq n$, there exist
\begin{enumerate}[(i)]
\item an open neighborhood $V$ of  $x$ in $X$, endowed with the induced filtration,
\item an open neighborhood $U\subseteq V$ of $x$ in $X_i\backslash X_{i-1}$, 
\item a compact filtered space, $L$,
of dimension $n-i-1$, where the open cone, $\rc L$, is provided with the {\em conical filtration}, $(\rc L)_{j}=\rc L_{j-1}$,
\item a homeomorphism, $\varphi \colon U \times \rc L\to V$, 
such that
\begin{enumerate}[(a)]
\item $\varphi(u,\tv)=u$, for each $u\in U$, 
\item $\varphi(U\times \rc L_{j})=V\cap X_{i+j+1}$, for each $j\in \{0,\ldots,n-i-1\}$.
\end{enumerate}
\end{enumerate}
The pair  $(V,\varphi)$ is a   \emph{conical chart} of $x$
 and the filtered space $L$ is a \emph{link} of $x$. 
The CS set $X$ is called  \emph{normal} if its links are connected.
\end{definition}

Perversity is the main ingredient in intersection homology, which allows 
a ``control of the deviation of transversality''  between simplexes and strata.
The original perversities introduced by Goresky and MacPherson in \cite{GM1} 
depend only on the codimension of the strata.
More general perversities  (\cite{MacPherson90,FriedmanBook, MR1245833, MR2210257}) 
are defined on the set of strata. 
They  are used in \cite{CST3} where we develop a topological invariance by refinement of the filtration,
in \cite{CST4}  with a  blown-up cohomology,
and in \cite{CST2,ST1} with a development of Poincar\'e duality for pseudomanifolds with a cap product.
We first review their definitions.
\begin{definition}\label{def:perversitegen}
A \emph{perversity on a filtered space, $X$,} is a map $\ov{p}\colon \cS_{X}\to \ov{\Z}={\Z}\cup \{\pm\infty\}$
taking the value 0 on the regular strata. The pair $(X,\ov{p})$
is called a \emph{perverse space}, or a \emph{perverse CS set} if $X$ is a CS set. 

A \emph{constant perversity} $\ov k$, with $k\in \ov \Z$, is defined by $\ov k(S)=k$ for any singular stratum $S$.
The \emph{top perversity}  $\ov t$ is
 defined by $\ov{t}(S)=\codim S-2$, if $S$ is a singular stratum. 
 Given a perversity $\ov p$ on $X$,  the \emph{complementary perversity} on $X$, $D\ov{p}$, is characterized 
 by $D\ov{p}+\ov{p}=\ov{t}$.

\smallskip
Any map $f\colon \N\to\N$ such that $f(0)=0$ defines a perversity $\ov{p}$ 
 by $\ov{p}(S)=f(\codim S)$. Such perversity is called \emph{codimensional}. 
In general, we denote by the same letter the perversity $\ov{p}$ and the map $f$.
The dual perversity of a codimensional perversity remains codimensional.
Among the codimensional perversities we find the original perversities of  \cite{GM1}.
\end{definition}

\begin{definition}\label{def:perversite}
A \emph{Goresky-MacPherson perversity (or GM-perversity)} is a map 
$\ov{p}\colon \{2,\,3,\dots\}\to\N$ such that  
$\ov{p}(2)=0$ and
$\ov{p}(i)\leq\ov{p}(i+1)\leq\ov{p}(i)+1$ for all $i\geq 2$. 
\end{definition}

\begin{definition}\label{def:perversiteimagereciproque}
Let $f\colon X\to Y$ be a stratified map and $\ov{q}$ be a perversity on $Y$.
 The  \emph{pull-back perversity of $\ov{q}$ by $f$} is the perversity $f^*\ov q$ on $X$ defined by
$(f^*\ov{q})(S)=\ov{q}(S^f)$,
for each $S\in \cS_{X}$.
For a canonical injection, $\iota$,  we still denote by $\ov{p}$ the perversity $\iota^*\ov{p}$ and call it the \emph{induced perversity.}
\end{definition}

The  product with a topological space, $X\times M$, is endowed with
the pull-back perversity of $\ov{p}$ by the canonical projection $X\times M\to X$, also denoted by $\ov{p}$.

Let $X$ be compact. The strata of the cone $\rc X$ are the singleton $\{\tv\}$ and the products
$S\times ]0,1[$ with $S$ a stratum of $X$. 
A perversity $\ov{q}$ on the open cone,  $\rc X$, induces a perversity on $X$, also denoted by $\ov{q}$ and defined by
$\ov{q}(S)=\ov{q}(S\times ]0,1[)$. 

%%%%%%%%%%%%%%%%%%%%%%%%%%%%%%%%
\section{Intersection homology}\label{sec:lesdeux}

Let $(X,\ov{p})$ be an $n$-dimensional perverse space.
Let us begin with some recalls from \cite{MR0350744}. 

%%%%%%%%%%%%
\subsection{Dimension of polyhedra}\label{subsec:dimpoly}
By definition, a subset $P\subset \R^n$ is a \emph{polyhedron} if each point $a\in P$ has a cone neighbourhood
$a*L\subset P$ with $L$ compact. 
A \emph{subpolyhedron} $Q$ of $P$ is a subset $Q\subset P$ which is itself a polyhedron.
Any polyhedron is a locally finite union of simplexes (\cite[Theorem 2.2]{MR0350744}), $P=\cup_{j}\Delta_{j}$,
called a triangulation of $P$.
If $Q$ is a subpolyhedron of $P$, then any triangulation of $Q$ can be extended to a triangulation of $P$.
The existence of triangulations allows a definition of dimension.

\begin{definition}\label{def:dimension}
The \emph{dimension of a polyhedron} $P$ is given by
$\dim P=\max \dim \Delta_{j}$ where the $\Delta_{j}$'s form a triangulation of $P$.
(This notion does not depend on the choice of the triangulation.)
A subspace $A\subset P$ of a polyhedron is of \emph{polyhedral dimension} less than or equal to $\ell$
 if  $A$ is included in a polyhedron $Q$ with $\dim Q\leq \ell$. 
 It is said of polyhedral dimension $k$ if $\dim A\leq k$ and $\dim A\nleq k-1$.
\end{definition}

This definition verifies
 \begin{equation}\label{equa:dimunion} 
 \dim (A_{1}\cup A_{2})= \max (\dim A_{1},\dim A_{2}).
 \end{equation}
With  this  definition, we  do   a selection among singular simplexes in the spirit of \eqref{equa:keyfirst}.

\begin{definition}\label{def:homotopygeom}
Let $(X,\ov{p})$  be a  perverse space.
A simplex $\sigma\colon \Delta\to X$ is  
\emph{$\ov{p}$-allowable}  
if, for each singular stratum $S$, the set $\sigma^{-1}S$ verifies
\begin{equation}\label{equa:admissible}
\dim \sigma^{-1}S\leq \dim\Delta-\codim S +\ov{p}(S)=\dim\Delta-2-D\ov{p}(S),
\end{equation}
with the convention $\dim\emptyset=-\infty$.
A singular  chain $\xi$ is \emph{$\ov{p}$-allowable} if it can be written as a linear combination of 
$\ov{p}$-allowable  simplexes,
and of \emph{$\ov{p}$-intersection} if $\xi$ and its boundary $\partial \xi$ are $\ov{p}$-allowable.
We denote by $C_{*}^{\ov{p}}(X;G)$ the complex of singular chains of $\ov{p}$-intersection
and by $H_*^{\ov{p}}(X;G)$ its homology, called
\emph{$\ov{p}$-intersection homology of $X$ with coefficients in a module $G$ over a Dedekind ring $R$}.
\end{definition}

\begin{remark}
We specify the notion of $\overline{p}$-allowable simplex, $\sigma\colon \Delta^k\to X$, for $k= 0, \,1,\, 2$, 
in function of the value of the perversity $D\overline{p}$ on each singular stratum.
This simplex $\sigma$ is $\overline p$-allowable if, for each singular stratum $S$,  the following conditions hold:

\begin{table}[h!]
  \centering
  \begin{tabular}{|c|c|c|c|c|}
    \hline
    \diagbox{k}{$D\overline p(S)$}    & -1 & 0  &$\geq 1$ \\  \hline
0   & $\sigma^{-1}(S) =\emptyset$ &  $\sigma^{-1}(S) =\emptyset$& $\sigma^{-1}(S) =\emptyset$\\[.1cm]     \hline
 1  &  $\sigma^{-1}(S)$ finite set&$\sigma^{-1}(S) =\emptyset$ & $\sigma^{-1}(S) =\emptyset$\\ [.1cm]      \hline
 2  & $\dim \sigma^{-1}(S)\leq 1$& $\sigma^{-1}(S)$ finite set & $\sigma^{-1}(S) =\emptyset$\\  [.1cm]     \hline
  \end{tabular}
\end{table}
For instance, if $D\ov{p}\geq 0$, which is the case of the GM-perversities,
a 0-simplex or a 1-simplex are $\ov{p}$-allowable if, and only if, $\sigma^{-1}S=\emptyset$ for any singular stratum $S$:
this means that they must lie in the regular part.
\end{remark}

\begin{example}\label{exam:R3filtre}
Consider two 3-simplexes of the ambient stratified space $X=\R^3$, with two singular strata,
$S_{0}=\{x_{0}=0\}$, $S_{1}=]0,\infty[$ and a regular stratum
$S_{2}=\R^3\menos [0,\infty[$.
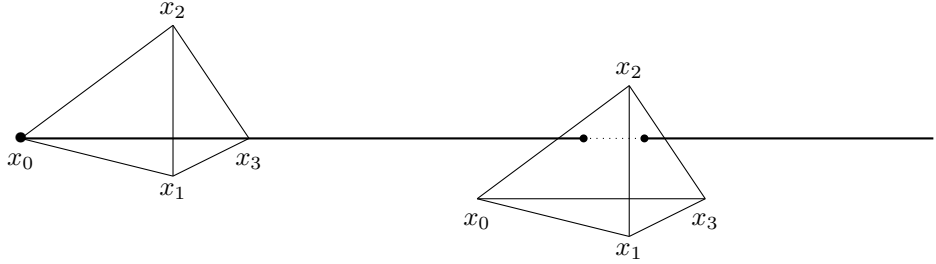
\medskip
\begin{center}

\begin{tikzpicture}

\draw [color=black] (-2,0.2) node {$x_0$};
\draw [color=black] (0,-0.2) node {$x_1$};
\draw [color=black] (0,2.2) node {$x_2$};
\draw [color=black] (1,0.2) node {$x_3$};

\draw [color=black] (4,-.6) node {$x_0$};
\draw [color=black] (6,-1) node {$x_1$};
\draw [color=black] (6,1.4) node {$x_2$};
\draw [color=black] (7,-.6) node {$x_3$};

\draw [color=black] (-2,0.5) node {$\bullet$};
	
\draw [color=black] (0,0)-- (1,0.5);
\draw [color=black] (1,0.5)--  (0,2);
\draw [color=black]  (0,2)-- (0,0);

\draw [color=black]  (0,0)-- (-2,0.5);
\draw [color=black]  (0,2)-- (-2,0.5);

\draw [color=black] (6,-.8)-- (7,-0.3);
\draw [color=black] (7,-.3)--  (6,1.2);
\draw [color=black]  (6,1.2)-- (6,-.8);

\draw [color=black,]  (4,-.3)-- (7,-.3);
\draw [color=black]  (6,-.8)-- (4,-.3);
\draw [color=black]  (6,1.2)-- (4,-.3);

\draw [color=black, thick]  (-2,0.5)-- (5.4,0.5);
\draw [color=black, dotted]  (5.4,0.5)-- (6.2,0.5);
\draw [color=black, thick]  (6.2,0.5)-- (10,0.5);
\fill [color=black] (5.4,0.5)  circle (1.5pt);
\fill[color=black] (6.2,0.5) circle (1.5pt);

\end{tikzpicture}
\end{center}

For the left-one, we have
$\dim \sigma^{-1}S_{0}=0$,
$\dim \Delta-2-D\ov{p}(S_{0})=1-D\ov{p}(S_{0})$,
and
$\dim \sigma^{-1}S_{1}=1$,
$\dim \Delta-2-D\ov{p}(S_{1})=1-D\ov{p}(S_{1})$.
This is a $\ov{p}$-allowable simplex for any perversity $\ov{p}$ such that $D\ov{p}(S_{0})\leq 1$ 
and
$D\ov{p}(S_{1})\leq 0$, thus in particular for $D\ov{p} \equiv 0$. 
The previous remark implies that the 1-simplex
 $[x_{0},x_{3}]$ 
 is not $\ov{p}$-allowable if $D\ov{p}\equiv 0$.
 
 Let us continue with the right-one simplex. Here, we have $\sigma^{-1}S_{0}=\emptyset$
 and $\dim \sigma^{-1}S_{1}=1$.
 We get a $\ov{p}$-allowable simplex if $D\ov{p}\leq 0$.
We can also check that any $\ell$-face of $\Delta$, for $\ell\in\{0,1,2\}$, is a $\ov{p}$-allowable 
simplex if $D\ov{p}\equiv 0$.
\end{example}

%%%%%%%%%%%%%%%%%%%%%
\subsection{Chain maps induced by stratified maps}
We begin with the chain map induced by a stratified map and the compatibility of this association with homotopy.  
We apply it to the  canonical projection $X\times \R\to X$.

\begin{proposition}\label{prop:stratifieethomotopie}
Let $f\colon (X,\ov{p})\to (Y,\ov{q})$ be  a stratified map between two perverse spaces with  $f^*D\ov q \leq D\ov p$.
The association $\sigma\mapsto f\circ \sigma$ sends a $\ov{p}$-allowable simplex on a $\ov{q}$-allowable simplex
and defines a chain map
$f_{*}\colon C_*^{\ov p} (X;G )\to C_*^{\ov q} (Y;G )$.
\end{proposition}

\begin{proof}
Let 
$\sigma\colon \Delta\to X$
be a $\ov{p}$-allowable  simplex.
Since $f_* (\partial\sigma) = \partial f_*(\sigma)$ it suffices to prove
that $f_*(\sigma)$ is a $\ov{q}$-allowable  simplex.
 By definition, the simplex $\sigma$ verifies 
\begin{equation}\label{equadansX}
\dim \sigma^{-1}S \leq \dim \Delta - D\ov{p}(S) -2,
\end{equation}
for each singular stratum $S \in \cS_{X}$, and we have to prove 
\begin{equation}\label{equadansY}
\dim \sigma^{-1}  f^{-1}  T \leq\dim\Delta - D\ov{q}(T) -2,
\end{equation}
for each  singular stratum  $T\in \cS_{Y}$.
Since $f$ is a stratified map, there exists a family of strata $\{S_i  \mid i\in I\} \subset \cS_X$ 
with $f^{-1}(T)=\sqcup_{i\in I}S_{i}$, with $\codim T \leq \codim S_{i}$ for each $i\in I$. 
In particular, the strata $S_i$ are singular strata. 
Since $\sigma(\Delta)$ is compact, the family $J = \{i \in I \mid  S_i \cap \sigma (\Delta) \ne \emptyset\}$ is finite.
We get
\begin{eqnarray*}
\dim \sigma^{-1} f^{-1} (T) &=& \dim \bigcup_{i\in J} \sigma^{-1}(S_i)
= \max\{ \dim \sigma^{-1}(S_i)\}  \\
&\leq_{(1)}& %
\dim \Delta-\min \{  D \ov{p}(S_{i}) \mid i \in J\} - 2
\\
&\leq_{(2)}& %
\dim  \Delta- D\ov q (T) -2,
\end{eqnarray*}
where %
(1) comes from \eqref{equadansX}
and (2) is a consequence of $f(S_i)\subset T$ and $ f^*D\ov q \leq D\ov p$.
\end{proof}

\begin{proposition}\label{prop:homotopesethomologie}
Let $\varphi\colon (X\times [0,1],\ov{p})\to (Y,\ov{q})$ be 
a homotopy between two stratified maps 
$f,\,g\colon (X,\ov{p})\to (Y,\ov{q})$ with 
$\varphi^*D\ov q \leq D\ov p$. Then 
 $f$ and $g$
induce the same map in homology,
$f_{*}=g_{*}\colon H_*^{\ov p} (X;G )\to H_*^{\ov q} (Y;G )$.
\end{proposition}

\begin{proof} 
From \propref{prop:stratifieethomotopie}, we get the homomorphism
$\varphi_{*}\colon H_*^{\ov p} (X\times [0,1] )\to H_*^{\ov q} (Y )$.
The canonical injections,
 $\iota_{0},\,\iota_{1}\colon X\to X\times [0,1]$,
  defined by $\iota_{k}(x)=(x,k)$ for $k=0,\,1$,  are stratified maps and induce homomorphisms
  $\iota_{0,*},\,\iota_{1,*}\colon  H_*^{\ov p} (X )\to H_*^{\ov p} (X\times [0,1] )$.
  Since 
$f=\varphi\circ\iota_{0}$ and $g=\varphi\circ\iota_{1}$, it suffices to prove 
$\iota_{0,*}=\iota_{1,*}$. 

\smallskip
Let $\sigma\colon \Delta=[e_{0},\dots,e_{m}]\to X$ be a simplex.
The vertices of the product 
$\Delta \times [0,1]$ are $a_{j}=(e_{j},0)$ and $b_{j}=(e_{j},1)$. We define an $(m +1)$-chain on   $\Delta \times [0,1]$ by
$P=%
\sum_{j=0}^m (-1)^j [ a_{0},\dots,a_{j},b_{j},\dots,b_{m}].$
This gives a chain homotopy, 
$h\colon
C_{*}(X)\to C_{*+1}(X\times [0,1])$,
between $\iota_{0,*}$ and $\iota_{1,*}$, defined by
$\sigma\mapsto (\sigma\times\id)_{*}(P)$. % 

\medskip
Now, it remains to prove that the image $h(\sigma)$ of a $\ov{p}$-allowable simplex is a
$\ov{q}$-allowable chain.
To any $j\in\{0,\dots,m\}$, we associate the simplex
$\tau_{j}\colon\nabla=[ v_{0},\dots,v_{m+1} ]\to \Delta^m\times [0,1]$,
defined by
$(v_{0},\dots,v_{m+1})\mapsto (a_{0},\dots,a_{j},b_{j},\dots,b_{m})$. 
By construction we have
$\Delta^m\times [0,1]=\cup_{j=0}^m\tau_{j}(\nabla)$.
For each stratum $S \in \cS_{X}$, we have
$
\dim \tau_{j}^{-1}(\sigma\times\id)^{-1}(S\times [0,1]) \leq 
\dim \sigma^{-1}S +1
\leq 
m- D \ov{p}(S)-1$.
We get
  \begin{eqnarray*}
\dim h(\sigma)^{-1} S
  &\leq &
  \max \left\{ \dim \tau_{j}^{-1}(\sigma\times\id)^{-1}(S\times [0,1])\mid  j \in \{0, \dots,m\}\right\}\\
  &\leq&  m- D \ov{p}(S)-1
  %
  %&= &
 =\dim\nabla- D \ov{p}(S \times [0,1]) -2,
   \end{eqnarray*}
   which implies the desired inclusion 
   $h(C_{*}^{\ov{p}}(X))\subset C_{*}^{\ov{p}}(X\times [0,1])$.
 \end{proof}

\begin{corollary}\label{cor:RfoisX}
Let $(X,\ov{p})$ be a perverse space. 
The inclusions $\iota_{z}\colon X\hookrightarrow \R\times X$, $x\mapsto (z,x)$,  with $z\in\R$, 
and the projection  $\pr \colon \R\times X\to X$ induce   isomorphisms
$H^{\ov{p}}_{k}(X;G)\cong H^{\ov{p}}_{k}(\R\times X;G)$.
\end{corollary}

%%%%%%%%%%%%%%%%%
\subsection{The case of a cone}

The space $X$ is identified with $X\times \{1/2\}\subset \rc X.$

\begin{proposition}\label{prop:conenewdim}
Let $X$ be a compact filtered space and $\ov{p}$ be a perversity on the cone $\rc X$ endowed with the cone filtration.
If we  also denote  by $\ov{p}$ the perversity induced on $X$, we have:
 $$H_{k}^{\ov{p}}(\rc X;G)=
 \left\{
 \begin{array}{cl}
  H_{k}^{\ov{p}}(X;G)&\text{ if } k\leq D\ov{p}(\tv),\\
 0&\text{if } 0\neq k> D\ov{p}(\tv),\\
 R&\text{if } 0= k> D\ov{p}(\tv),
 \end{array}\right.
 $$
  where $\tv$ is the apex of the cone $\rc X$.
\end{proposition}

\begin{proof} 
The allowability condition for the stratum $\{\tv\}$ of a simplex 
$\sigma\colon \Delta^\ell \to \rc X$ is:
\begin{equation}\label{equa:encoreadmis}
\dim \sigma^{-1}\tv\leq \ell-D\ov{p}(\tv)-2.
\end{equation}

1) Thus, for any $\ell\leq D\ov{p}(\tv)+1$, we have  $\sigma^{-1}\tv=\emptyset$
and there is an isomorphism 
$C^{\ov{p}}_{\leq D \ov{p}(\tv) +1}(\rc X)\cong C^{\ov{p}}_{\leq D\ov{p}(\tv)+1 }(X\times ]0,1[)$.
This gives the first part of the statement (cf. \corref{cor:RfoisX}).

\medskip
2) For proving  the rest of the statement, we show that the inclusion $C_{*}^{\ov{p}}(\rc X)\hookrightarrow C_{*}(\rc X)$
induces an isomorphism in homology in degrees $\ell \geq D\ov{p}(\tv)+1$.  
Thus, let $ \sigma\colon \Delta^\ell\to \rc X$ be a $\ov{p}$-allowable simplex. If $[x,t]\in\rc X$, we set $s\cdot [x,t]=[x,st]$ for $s\in [0,1]$. This gives a map
$ \tc\sigma\colon \Delta^{\ell+1} = \{\tw\}*\Delta^\ell\to \rc X$,
$\tc\sigma(sx+(1-s)\tw)=s\cdot \sigma(x)$.
We first show that $\tc \sigma$ is $\ov{p}$-allowable. For this, we distinguish between the two types of strata.

a) Let $S\times ]0,1[$ with $S$ a stratum of $X$. Suppose $(\tc\sigma)^{-1}(S\times ]0,1[)\ne\emptyset$.
From $ (\tc\sigma)^{-1}(S\times ]0,1[)=\sigma^{-1}(S \times ]0,1[) \times ]0,1[$, we deduce
$\sigma^{-1}(S\times ]0,1[)\ne\emptyset$ and
$$
\dim (\tc\sigma)^{-1}(S\times ]0,1[)
=
1+ \dim \sigma^{-1}(S\times ]0,1[) \leq
1+\ell - D \ov{p}(S\times ]0,1[) - 2 .
$$

b)  We have $(\tc \sigma)^{-1}(\tv)=\tw*\sigma^{-1}(\tv)$.   There are two situations.\\
-- If $\sigma^{-1}(\tv)\ne\emptyset$, then
$\dim (\tc\sigma)^{-1}(\tv) \leq  1+\dim \sigma^{-1}(\tv)\leq   1+\ell- D\ov{p}(\tv) - 2$.\\
-- If $\sigma^{-1}(\tv)=\emptyset$, the $\ov{p}$-allowability condition of $\tc \sigma$ is
$$
\dim (\tc\sigma)^{-1}(\tv)
=
  \dim\tw
 = 0 \leq \ell -D\ov{p}(\tv) - 1 =
\ell+1-D\ov{p}(\tv) - 2 .
$$
The isomorphism in homology now derives from
$$\partial\,c\sigma=\left\{
\begin{array}{cl}
\sigma-c\partial\sigma 
&
\text{if}\; \ell\neq 0,\\
\sigma-\tv
&
\text{if}\; \ell= 0.
\end{array}\right.
$$
\end{proof}

%%%%%%%%%%%%%%%%%%%%%%%%%%%%%
\section{Strong general position}\label{sec:strong}
If $\sigma\colon \Delta\to X$ is a singular simplex, the subset $\sigma^{-1}S$ can be very wild.
We introduce a simplicial notion that can replace it.

\begin{definition}\label{def:envl}
Let $\sigma \colon \Delta \to X$ be a singular simplex on a filtered space $X$ and $S \in \cS_X$ be a stratum.
A \emph{simplicial envelope} of $\sigma^{-1}S$ is a finite family $\cT_{\sigma,S} $ of Euclidean simplexes
 such that
$$
\sigma^{-1}S \subset \bigcup_{T \in \cT_{\sigma,S}}T
$$  
\emph{and} $\dim \sigma^{-1}S = \max  \left\{ \dim T \mid  T \in \cT_{\sigma,S} \right\}$. 
If $\sigma^{-1}S = \emptyset$ we set $\cT_{\sigma,S} = \emptyset$.
A \emph{simplicial system for $\sigma$} is a  family 
$\cT_\sigma = \left\{ \cT_{\sigma, S} \mid S \in \cS_X\right\}$
 where each $\cT_{\sigma, S}$ is a simplicial envelope of $\sigma^{-1}S$.
 A \emph{simplicial system for $X$} is a family $\cT=\left\{\cT_{\sigma}\mid \sigma\in\sing\, X\right\}$
 of simplicial systems for the simplexes of $X$. 
\end{definition}
 
Simplicial envelopes and systems  always exist.
Given a simplex $\sigma\colon \Delta\to X$, a stratum $S$ and a simplicial system $\cT_{\sigma, S}$,
the allowability condition \eqref{equa:admissible} is equivalent to the following inequalities, 
for each $T \in \cT_{\sigma, S}$,
\begin{equation}\label{equa:admissible2}
\dim T\leq  \dim \Delta -2-D\ov{p}(S).
\end{equation}
%

 %%%%%%%%%%%
 \subsection{General position}\label{subsec:general}
 Detecting if a simplex included in an allowable simplex is still allowable is a crucial point 
 and  moving to general position will be a useful tool.
Let us recall some basic facts, sending to \cite{MR0350744} for more information.

\medskip
It is well known that two  affine spaces $E,F \subset \R^m$, with $E\cap F\neq\emptyset$, are in general position if 
 $\dim(E\cap F) = \dim E + \dim F - m$. 
 General position is generic in the sense that it is always possible to slightly move two affine spaces in a general position.
 This is also a stable situation in the sense that if  two affine spaces  are in general position and we slightly move one of them, then they remain in a general position. 
 
 \medskip
 The situation is different for polyhedra. Recall that two  polyhedra $\{P,Q\}$ included in a simplex $\Delta$, are in general position if 
\begin{equation}\label{equa:generalposition}
\dim (P\cap Q) \leq \dim P + \dim Q - \dim \Delta.
\end{equation}
Given two polyhedra it is always possible to slightly move one of them in order to get the general position
 (cf. \cite{MR0350744}) and thus general position of polyhedra is generic. 
 But this notion is not stable as shows the following example.
 
 \begin{example}\label{exam:passtable}
 Let $\dim\Delta=2$ and $a,b,c \in \Delta$ be three distinct  points aligned. Then, the 1-simplexes $P = [a,b]$ and $Q = [b,c]$ are in general position.
 But the simplexes $\{P,Q_\eps = [\eps a + (1-\eps)b,c]\}$ are not in general position, for any $\eps\in]0,1]$. 
 There is no stability.
  \end{example}

With this lack of stability, general position is not the right situation for our purpose. 
Thus, we  introduce in \subsecref{subsec:STP} a more robust notion called \emph{strong general position.}
Before that, we show that the two situations of general position for affine spaces 
and polyhedra coincide with an extra hypothesis.

\begin{lemma}\label{lem:cleim}
Let   $\{P,Q\}$ be two simplexes living in a simplex 
$\Delta$ and such that $\inte P \cap \inte Q \ne \emptyset$.
If $\R^m$, $E$, $F$, $H$ denote  the affine spaces generated by  $\Delta$, $P$, $ Q $ and $P \cap Q$  respectively,
we have $H=E\cap F$. Moreover, the following statements are equivalent.
\begin{enumerate}[i)]
\item The affine spaces $E,\,F$ are in general position in $\R^m$.
\item The simplexes $P,\,Q$ are in general position in $\Delta$.
\item $\dim (P \cap Q)  =  \dim P +\dim Q - \dim \Delta$.
\end{enumerate}
\end{lemma}

\begin{proof}
We fix a point  $x_{0}\in \inte P \cap \inte Q$. 
 We clearly have  $H\subset E \cap F$. If $E\cap F =\{x_{0}\}$ then we get 
$H= E \cap F$. So, we can suppose that $E\cap F$ contains some $z\ne x_{0}$.
Since $E$ is the affine space generated by the simplex $P$
and  $x_{0}\in\inte P$, there exists 
 $x\in \inte P\,\cap\,  ]x_{0},z[$.
 Similarly, there exists
$y\in  \inte Q \,\cap\, ]x_{0},z[$ and we may suppose $x=y$.
Then,   $x\in\inte P \cap \inte Q \subset H$ which  implies $z \in H$. We have proven 
$E\cap F\subset H$ and the first assertion.
The second assertion is a consequence of the equality between the dimensions of a polyhedron and of its associated affine space.
\end{proof}

Finally, we link the notions of general position and admissibility. 

\begin{proposition}\label{prop:PosGen}
Let  $(X,\ov{p})$ be a perverse  space
and $\sigma \colon \Delta \to X$ be a $\ov{p}$-allowable simplex.
We fix a simplicial system $\cT_\sigma$ and a simplex $\nabla \subset \Delta$.
If, for any singular stratum $S$ of $X$ and any $T\in \cT_{\sigma,S}$, the simplexes
$\{T, \nabla\} $ are in general position in $\Delta$, then the restriction $\sigma_\nabla \colon \nabla \to X$ of $\sigma$ is 
a $\ov{p}$-allowable simplex. 
\end{proposition}

\begin{proof}
We have to prove  
$
\dim \sigma^{-1}_\nabla S = \dim (\sigma^{-1} S \cap \nabla)  \leq \dim \nabla - D\ov p(S)-2$.
By \eqref{equa:dimunion} and the definition of $\cT_{\sigma,S}$, it suffices to prove that
$
\dim (T \cap \nabla) \leq \dim \nabla - D\ov p(S)-2
$, for each $T \in \cT_{\sigma,S}$.
Applying \eqref{equa:generalposition} and \eqref{equa:admissible2}, we get 
\begin{eqnarray*}
\dim (T  \cap \nabla) 
&\leq&
\dim T
+
\dim \nabla
-
\dim \Delta
\leq
\dim \Delta -D\ov{p}(S)-2
+
\dim \nabla
-
\dim \Delta\\
&
\leq &
\dim \nabla -D\ov{p}(S)-2.
\end{eqnarray*}
\end{proof}

%%%%%%%%%%%%%%%%%%%%%%%
\subsection{Strong general position}\label{subsec:STP}
We start with a simplex $\Delta$ together with two simplexes, $T\subset \Delta$ and $V\subset \partial \Delta$.
In order to construct a pseudo-barycentric subdivision in \secref{sec:EPBS},
we have to find out a point $u\in \inte \Delta$ such that the simplex $T$ and the cone $\tc_u V$ are in general position.
In the main, the choice of the point $u$ requires the notions of genericity and stability.

\begin{definition}\label{def:genericstable}
Let $\Pi$ be a property defined on each point of the interior $\inte\Delta$ of  an Euclidean simplex $\Delta$. We set 
$$\Delta_\Pi = \{ u\in \inte\Delta \mid \Pi(u) \text{ is true} \}.
$$
The property $\Pi$ is  \emph{generic} if $\Delta_\Pi$ is a dense subset of $\inte\Delta$
and \emph{stable} if $\Delta_\Pi$ is an open  subset of $\inte\Delta$.
 \end{definition}
 
 As a dense subset of $\inte \Delta$ cannot be empty, the genericity condition implies the existence of points $u$
 with $\Pi(u)$ true.
 The stability condition means that we can move the choice of such point $u$ in in a neighborhood, and we will
 use that possibillity in the inductive step.
 
\medskip
The following notion of strong general position   is both generic and stable,
as we show in \propref{prop:genericstable}. 

\begin{definition} \label{def:strongGP}
Let $\Delta$ be an Euclidean simplex. We consider two Euclidean simplexes
$T \subset \Delta$ and $V \subset \partial \Delta $ and a point  $u \in \inte \Delta$. 
Denote by $\tc_u V =u* V$ the cone on $V$ of apex $u$.
The simplexes $\{T,\tc_u V\}$ are in a \emph{strong general position} if 
\begin{equation}\label{SGP}
 T \cap \tc_u V \subset \partial \Delta 
\ \ \ \hbox{ or } \ \ \ 
\left\{
\begin{array}{l}
\inte T \cap ( \tc_u V)^\circ \ne \emptyset , \text{ and }\\ 
\dim(T \cap \tc_u V ) =  \dim T + \dim V +1-\dim \Delta.
\end{array}
\right.
\end{equation} 
\end{definition}

This property is denoted  $\crP(u,T,V)$.  
We start with a study of the two properties appearing in \defref{def:strongGP}. 
Notice that $ T \cap \tc_u V \cap \partial \Delta = T \cap  V $.

\begin{lemma}\label{lem:twoproperties}
Let $\Delta$ be an Euclidean simplex. We consider two Euclidean simplexes
$T \subset \Delta$ and $V \subset \partial \Delta $ and a point  $u \in \inte \Delta$. 
\begin{itemize}
\item[a)] If
$
 T \cap \tc_u V  \menos \partial \Delta \ne \emptyset
 $,
then $u$ belongs to the closure of $\Delta_{\Pi_1}$, where $\Pi_1$ is defined by 
$$
\Pi_1(b)= ``\inte T \cap (\tc_b V )^\circ\ne \emptyset.\text{''}
$$ 
\item[b)] The following property $\Pi_2$ is stable,  where
$$
\Pi_2(b)=``\inte T \cap (\tc_b V )^\circ\ne \emptyset\quad \text{and}\quad  \dim( T \cap \tc_b V ) =  \dim T+ \dim V +1-\dim \Delta.\text{''}
$$
 \end{itemize}
 \end{lemma}
 
 \begin{proof}
 a) Let $z \in  T \cap \tc_u V  \menos \partial \Delta$ and $\eps >0$.
 We have to find $\omega \in B(u,\eps)$ such that $\inte T\cap {(\tc_{\omega}V)^\circ}\neq \emptyset$. %
We make out three cases.
 
 \smallskip
 i) If $ z \in \inte T  \cap (\tc_u V )^\circ $, we take  $w=u$.
 
 \smallskip
ii) Suppose $ z \in \partial T \cap   (\tc_u V)^\circ$.
Since  $ z \in (\tc_u V)^\circ$, then there exist $\lambda \in ]0,1[$ and $v \in V \menos \partial V$ such that  $z = \lambda u +(1-\lambda) v$. 
Since $T\neq \partial T$, there exists $t\in \inte T$.
 Let $\vec \alpha=\overrightarrow{zt}$ be the  vector determined by 
 $z=t+\vec\alpha$ in the associated affine space.
 We choose  $\eps'\in ]0,1[$  small enough to assure that  $w= u +\eps' \, \vec{\alpha }  \in B(u,\eps)$.
 The point
 $
 \lambda \eps'  t +(1-\lambda \eps')  z 
 =
 z+\lambda \varepsilon'\vec\alpha
 =
  \lambda \omega + (1-\lambda)v
  $
belongs to  $\inte T\cap (\tc_{\omega} V )^\circ$ since $\lambda ,\eps' \ne 0,1$ and $v\in \inte V$. 
We have acquired $\Pi_1(\omega)$.
  
\smallskip
 iii) Suppose $z \in  T \cap   \partial \tc_u V  $. 
 Since $ z \in \tc_u V \menos \partial \Delta$, then $z \not\in  \tc_u V \cap \partial \Delta =V$. 
On the other hand, from $z \in  \partial \tc_u V $, we deduce $z \in \tc_u \partial V \menos \partial V $.
 This gives $z = \lambda u +(1-\lambda) v $ with  $v\in \partial V$ and $\lambda \ne 0$.
 Let $b\in \inte V$.  Let us denote $\vec \alpha=\overrightarrow{bz}$  the vector  determined by 
 $z=b+\vec\alpha$.
 %
%Let  $\vec \alpha$ be the vector $z - b$ where $b$ is a  point of $ \inte V$.
We choose  $\eps'\in ]0,1[$  small enough to assure that  $w= u +\eps' \, \vec{\alpha }  \in B(u,\eps) $.
Set $b'=\frac{1-\lambda} {1+\lambda\eps'-\lambda} v +  \frac{\lambda\eps'} {1+\lambda\eps'-\lambda}b $.
We also have,
$$
z =  \frac \lambda{1+\lambda \eps'} w + \left( 1-  \frac \lambda{1+\lambda \eps'}   \right) b'.
$$
The point $b'$ does not belong to $\partial V$ because $\lambda \eps'\ne0$. 
Notice that $z \in( \tc_\omega V)^\circ$ since $ \frac \lambda{1+\lambda \eps'} \ne 0,1$.
We deduce $T \cap ( \tc_w V )^\circ \ne \emptyset$ and it  suffices to apply i) or ii).%

\smallskip
b) Let  $u \in  \inte \Delta$ with  $\inte T  \cap (\tc_u V )^\circ\ne \emptyset$ and 
$\dim (T \cap \tc_u V ) =  \dim T + \dim V+1- \dim \Delta$.
We have to find $\eps>0$ such that for any point $\omega \in B(u,\eps)$ we have $\Pi_2(\omega)$.
Let $\{t_1,\dots,t_p\}$ be a basis of $T$ and  $\{v_1,\dots,v_q\}$ be a basis of $V$. 
We choose $t_0\in \inte T \cap (\tc_u V)^\circ \ne \emptyset$ that we write
\begin{equation}\label{2370}
t_0 = \sum_{j=1}^p \mu_j t_j = \lambda_0 u +  \sum_{i=1}^q \lambda_i v_i .
\end{equation}
Since  $t_0 $ does not belong to the boundary,  the coefficients are in $]0,1[$. 
The generated affine spaces 
being in general position (\lemref{lem:cleim}), for any point $\omega \in B(u,\eps)$ and any $\eps >0$, we have
\begin{equation}\label{2375}
\omega = u + \sum_{j=1}^p B_j \overrightarrow{t_0t_j} +    \sum_{i=1}^q A_i \overrightarrow{t_0v_i} .
\end{equation}
Denote $A= A_1 + \dots + A_q$ and $B =B_1 +\dots+B_p,$.
We consider the following two points:
\begin{enumerate}[(i)]
\item $ y = \alpha_0 \omega + \sum_{i=1}^q \alpha_i v_i$,  with 
$
\alpha_0 =\lambda_0/(1 - A \lambda_0)$,  
$ \alpha_i =(\lambda_i - A_i \lambda_0)/(1 - A \lambda_0), \;i\in \{1, \dots,q\}$,
\item $ s = \sum_{j=1}^p \beta_j t_j$,
with 
$
\beta_j = \mu_j (1 - B\alpha_0) +B_j \alpha_0 ,\; j\in \{1, \dots,p\}$.
\end{enumerate}
If we choose $\eps>0$ small enough then both combinations are convex combinations
 verifying $s \in \inte T $ and $y  \in (\tc_\omega V )^\circ $.
 We are going to prove $y=s$ which implies
 \begin{equation}\label{2403}
\inte T \cap  (\tc_\omega V )^\circ\ne \emptyset.
\end{equation}
The equality $y=s$ is equivalent to the equality $\overrightarrow{t_0y} = \overrightarrow{t_0s}$.
A straightforward calculation gives the claim using the following properties:
\begin{itemize}
\item $\displaystyle 
  \vec 0 = \sum_{j=1}^p \mu_j \overrightarrow{t_0 t_j }= \lambda_0 \overrightarrow{t_0 u} +  \sum_{i=1}^q \lambda_i \overrightarrow{t_0 v_i}
  $ (cf.  \eqref{2370}),
\item 
$
\displaystyle
\overrightarrow{t_0 \omega} = \overrightarrow{t_0 u} + \sum_{j=1}^q B_j \overrightarrow{t_0t_j} +    \sum_{i=1}^q A_i \overrightarrow{t_0v_i} 
$
(cf. \eqref{2375}).
\end{itemize}
 It remains to prove $\dim (T \cap \tc_\omega V)  =  \dim  T + \dim V +1- \dim  \Delta$. 
 With the notation of  \lemref{lem:cleim}, the hypothesis 
 $\dim(T \cap \tc_u V ) =  \dim T + \dim V +1- \dim \Delta$ implies   $\R^m =E+H$. 
 This is equivalent to say that the family of vectors,
 $
 \left\{ \overrightarrow{t_0 t_1}, \dots, \overrightarrow{t_0 t_q}, \overrightarrow{t_0 v_1}, \dots, \overrightarrow{t_0 v_q}, \overrightarrow{t_0 u} \right\},
 $
 is of maximal rank $m$. 
 This means that some determinant  does not vanish and it remains true if we replace the involved points by close points.
 So, if $\eps>0$ is small enough we can suppose 
  \begin{equation}\label{rang}
\rank  \left\{ \overrightarrow{y t_1}, \dots, \overrightarrow{y t_q}, \overrightarrow{y v_1}, \dots, \overrightarrow{y v_q}, \overrightarrow{y \omega} \right\} = m.
 \end{equation}
 Let us come back from vector spaces to affine spaces:
 let $\R^m$, $E'$, $F'$, $H'$  be the affine spaces generated by  
 $\Delta$, $T$, $ \tc_\omega V $ and $T \cap \tc_\omega V$  respectively. 
Notice that $\dim T = \dim  E', \dim  \tc_\omega	 V = \dim  F'  $ and $\dim (T \cap \tc_\omega V)=\dim H'$.  
Moreover \lemref{lem:cleim} implies $H'=E'\cap F'$.
Equality \eqref{rang} gives $\R^m = E' + F'$. 
We deduce
$
m = \dim E' + \dim F' - \dim (E'\cap F') = \dim  E' +  \dim F' - \dim H'
$,
 and, therefore,
 $ \dim ( T \cap \tc_\omega V )  =  \dim T + \dim V +1- \dim \Delta$.  
\end{proof}

\begin{remark}\label{rem:genez}
Properties $\Pi_1$ and $\Pi_2$ of \lemref{lem:twoproperties} are not generic. 
Let us see an example with
$\Delta = <a_0,a_1,a_2>$, the 2-simplex generated by the points  $\{a_{0},a_{1},a_{2}\}$. 
We denote by $u$ the barycentre of $\Delta$. 
Let us choose $T=<a_0,u>$ and $V = <a_0>$, which gives $T= \tc_uV$. 
If we slightly move the vertex $u$, we get three possibillities for  $T \cap  \tc_\omega V$,
namely $ <a_0>$, $<a_0,\omega>  $ or $ <a_0,u>$. 
The first case gives $\inte T  \cap (\tc_\omega V )^\circ =\emptyset$ while the other two correspond to 
$\dim( T \cap \tc_\omega V ) = 1 \ne 0 = \dim T + \dim V +1-\dim \Delta$.
This motivates the  strong general position concept.
\end{remark}

\begin{proposition}\label{prop:genericstable}
Let $\Delta$ be an Euclidean simplex and 
$T \subset \Delta$, $V \subset \partial \Delta $ be two other Euclidean simplexes.
Then the property  $\crP(-,T,V)$ of \defref{def:strongGP} is generic and stable.
\end{proposition}

\begin{proof}
$\bullet$ \emph{First step: The stability.}
We prove that any point   $u\in \Delta_{\crP(-,T,V)}$ 
is an interior point.  
We distinguish three cases.

a)  $ T \cap \tc_u V =\emptyset$.
Since $T$ and $\tc_u V$ are compact subsets, the distance of $T$ to $\tc_{u} V$ is strictly positive: $\dist(T, \tc_u V ) =\eta >0$.  
Therefore, it suffices to prove that
$
 T \cap \tc_\omega V  =\emptyset
 $
 if $\dist(u,\omega)< \eta$.
 If
$
 T \cap \tc_\omega V  \ne \emptyset
 $
then there exists  $\lambda \in [0,1]$ and $v \in V$ with  $t=\lambda \omega +(1-\lambda) v \in  T \cap \tc_\omega V $. 
Let
$z =  \lambda u +(1-\lambda) v$. We obtain the contradiction
$$
\eta = \dist(T,\tc_u V) \leq \dist(t,z) = \lambda \,\dist(u,\omega) <  \eta.
$$

b)  $ T \cap \tc_u V \ne \emptyset$ and $ T \cap (\tc_u V) \menos \partial \Delta =\emptyset$.  
Let $t_0 \in  T \cap \tc_u V =  T \cap V$ be a base point.
 We consider  the \emph{infinite} open cone generated by  $T$,
  $\widehat T = \{ (1-\lambda) t_0 +\lambda t \mid  t \in T, t\ne t_0,  \lambda \geq 0\}$. First, we prove
\begin{equation}\label{2500}
\widehat T \cap \tc_u V \menos \partial \Delta = \emptyset.
\end{equation}
Let us suppose  there is a point  $z =  (1-\lambda) t_0 + \lambda t  \in \widehat T \cap \tc_u V \menos \partial \Delta $ with  $t \in T, t\ne t_0$ and $  \lambda \geq 0$.
Notice that 
 $\lambda >0 $ and  $t\in T\menos \partial \Delta$.
There are two cases:
\begin{itemize}
\item[+] $\lambda =  1$ which corresponds to $z=t  \in T\cap (\tc_uV) \menos\partial \Delta$ and which is impossible.
\item[+] $\lambda\ne  1 $. 
As $\widehat{T}$ is the space of the points of half-lines generated by $t_{0}$ and a point $t\in T$,
the three points $t, t_0,z$ are collinear, with  $t_0$ ``on the left.''  We thus have two situations

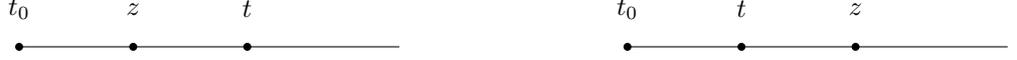
\begin{center}

\begin{tikzpicture}

\fill  (0,0)  circle (1.5pt);
\fill  (1.5,0)  circle (1.5pt);
\fill  (3,0)  circle (1.5pt);

\draw  (0,.5)  node {$t_0$};
\draw  (1.5,0.5)  node {$z$};
\draw  (3,0.5)  node {$t$};

\fill  (8,0)  circle (1.5pt);
\fill  (9.5,0)  circle (1.5pt);
\fill  (11,0)  circle (1.5pt);

\draw  (8,.5)  node {$t_0$};
\draw  (9.5,0.5)  node {$t$};
\draw  (11,0.5)  node {$z$};

\draw  (0,0)-- (5,0);
\draw  (8,0)--  (13,0);

\end{tikzpicture}
\end{center}

The left one implies $z\in T \cap (\tc_uV )\menos\partial \Delta =\emptyset$ 
and the right one $t\in T \cap (\tc_uV )\menos\partial \Delta =\emptyset$, which is impossible.
\end{itemize}
This gives the claim \eqref{2500}.
Let us notice that the property,
\begin{equation}\label{kleim}
\hbox{
``there exists $\eps>0$ such that for any  $\omega \in B(u,\eps)$ we have $
\widehat T \cap \tc_\omega V \menos \partial \Delta = \emptyset,\text{''}
$
}
\end{equation}
implies $ T \cap (\tc_w V) \menos \partial \Delta =\emptyset$ and  $u$ is an interior point of $\Delta_{\crP(-,T,V)}$.
We are reduced to prove \eqref{kleim} and, for that,  we use an induction on $\dim V$.

$\bullet$ $\dim V=0$. We choose $\eps = d(u,\widehat T)$, which is strictly positive since $u \not \in \widehat T$ (cf. \eqref{2500}). 
Let $\omega \in B(u,\eps)$ and suppose  there exists 
$y  \in \widehat T \cap \tc_\omega  V \menos  \partial \Delta $.
If we find a contradiction, we get  the claim $\widehat T \cap \tc_\omega V  \menos  \partial \Delta =\emptyset $.
We have $V = \{t_0\}$. 
Since $y \in \tc_\omega V \menos \partial \Delta$ then
 $y = \lambda t_0 + (1-\lambda) \omega$ with $\lambda \in [0,1[$. 
 Now, since $y \in \widehat T$, we have 
$y = (1-\mu)t_0 + \mu t$ with $\mu \geq 0, t\in T$ and $t\ne t_0$.
This gives 
$$
\omega = \left(1-\frac \mu{1-\lambda}\right) t_0 +  \frac \mu{1-\lambda} t \in \widehat T.
$$
The contradiction sought is:
$
\eps =  \dist(u,\widehat T) \leq   \dist(u,\omega) < \eps.
$

\smallskip
$\bullet$ Inductive step.
Let  $V\subset \partial \Delta$ with  $\dim V\geq 1$.
Let us suppose that  \eqref{kleim} is proven for each face of  $\partial V$. So, there exists $\eps >0$ with 
$\widehat T \cap \tc_\omega(  \partial V) \menos \partial \Delta = \emptyset$ for each $\omega \in B(u,\eps)$.
We prove again by the absurd, assuming the existence of
$y  \in \widehat T \cap \tc_\omega  V \menos  \partial \Delta $.
A contradiction will give the claim $\widehat T \cap \tc_\omega V  \menos  \partial \Delta =\emptyset $.
From $t_0\in \partial \Delta$ and $y\not \in \partial \Delta$, we deduce $t_0\ne y$. We distinguish two cases:\\
$i)$ $y\in  \partial \tc_\omega V$ which gives $y \in \widehat T \cap ( \partial \tc_\omega V)\menos \partial \Delta = 
 \widehat T \cap  \tc_\omega (\partial V) \menos \partial \Delta =\emptyset$. \\
 $ii)$  $y\not\in  \partial \tc_\omega V$. The intersection of the simplex $\tc_\omega V$ with the ray
 $[t_0,y,\infty[ = \{(1- \lambda) t_0 +\lambda y \ / \ \lambda \geq 0\}$ 
 contains the simplex $[t_0,y]$.
Thus there exists $y'\in \partial \tc_\omega V \cap[ t_0,y,\infty[ $
and the  three points $t_0, y,y'$ are collinear in this order.
 From $t_0,y\in \widehat T$, we get $y'\in \widehat T$. From $t_0\in \partial \Delta$ and $y\not\in \partial \Delta$,
 we deduce  $y'\not\in \partial \Delta$. 
 As in the first case, the contradiction comes from
  $y' \in \widehat T \cap ( \partial \tc_\omega V)\menos \partial \Delta = 
 \widehat T \cap  \tc_\omega (\partial V) \menos \partial \Delta =\emptyset$.

\smallskip
c) $
 T \cap \tc_u V \menos \partial \Delta \ne \emptyset.
 $
By \defref{def:strongGP}, we have 
$
\inte T\cap (\tc_u V )^\circ\ne \emptyset 
$
and $
\dim(T \cap \tc_u V ) =  \dim T+ \dim V +1- \dim \Delta.
$
Now, it suffices to apply \lemref{lem:twoproperties} b).

\medskip
$\bullet$ \emph{Second step: The genericity.}
For any point   $u\in \Delta^\ell \menos \partial \Delta^\ell$ and for any $\eps>0$, we show the existence of $\omega \in B(u,\eps)$ 
such that $\omega \in \Delta_{\crP(-,T,V)}$.
If $ T \cap \tc_u V \menos \partial \Delta = \emptyset$, it suffices to choose $\omega=u$.
Let us suppose 
 $ T \cap \tc_u V \menos \partial \Delta \ne \emptyset$. 
 Following \lemref{lem:twoproperties} a), there exists $\omega' \in B(u,\eps/2)$ such that
  $
\inte T \cap (\tc_{\omega'} V)^\circ \ne \emptyset $.
With the notation of  \lemref{lem:cleim}, we have  $E+F =  \R^m$ or $E+F \subsetneq  \R^m$.
In the first case we get $\dim(T \cap \tc_{\omega'} V ) =  \dim T + \dim V +1-\dim \Delta$.
This gives $\omega' \in \Delta_{\crP(-,T,V)}$. We take $\omega=\omega'$.

For the second case, let us consider a vector $0\ne\vec\alpha \in \R^m \menos (E+F)$ and take $\eps'>0$ small enough to get 
$\omega = \omega' + \eps' \vec \alpha \in B(u,\eps)$. We get the claim $\omega \in \Delta_{\crP(-,T,V)}$ if we prove
  $ T \cap \tc_{\omega} V \subset \partial \Delta$.
  
 A point $t \in  T \cap \tc_{\omega} V$ is of the form $ \lambda \omega + (1-\lambda) v$ where $\lambda\in [0,1] $ and $v \in V$. 
   Let $t_0 \in \inte T \cap (\tc_{\omega'} V)^\circ$. We have, 
   $ \lambda \eps' \vec\alpha =\overrightarrow{t_0 t}   - (1-\lambda) \overrightarrow{t_0v}
    -  \lambda \overrightarrow {t_0\omega'} \in E+F$ 
    and therefore $\lambda =0$. This gives $t =v \in V \subset \partial \Delta$.
\end{proof}

%%%%%%%%%%%%%%%%%
\section{Pseudo-barycentric subdivisions}\label{sec:EPBS}

Given a singular simplex $\sigma \colon \Delta \to X$ of a filtered space $X$, a pseudo-barycentric subdivision $\cB$ of $\Delta$ is a subdivision 
similar to the barycentric subdivision, except the fact that the new vertices are not barycentres but close points of them. 
These points are chosen  to control the relative position of the simplexes of the triangulation $\cB$ and the simplicial envelope 
 (see \defref{def:envl}) 
of the singular part $\sigma^{-1}\Sigma_X$.

 \begin{definition} 
A \emph{triangulation system} of a space $X$ is a family $\cB = \{\cB_\sigma \mid \sigma\colon \Delta_{\sigma}\to X \}$ of triangulations 
$\cB_\sigma$  of $\Delta_{\sigma}$. %   
The \emph{diameter }of a triangulation is the maximum of the diameters of its simplexes.
 \end{definition}

\begin{definition}\label{defB}
Let $X$  be a filtered space, $\cT$ be a simplicial system  and  
$\cP = \{u_\sigma \in \inte\Delta_{\sigma} \mid \sigma \in \sing\,X\}$ 
be a family of points, called \emph{pseudo-barycentres}.
 A \emph{$\cT$-pseudo-barycentric system}, $\cB$, of $X$, 
 associated to $\cP$, is a triangulation system verifying the properties
 (PB1)-(PB5), for each 
  $\sigma\colon \Delta^\ell \to X$: 
\begin{enumerate}[(PB1)]
\item $u_{\sigma}\in\inte \Delta_{\sigma}$. 
\item For any $i$-th face 
$\partial_{i}\sigma\colon \partial_{i}\Delta_{\sigma}\to X$ of $\sigma$,
we have 
$\cB_{\partial_{i}\sigma} =
(\cB_\sigma)\cap \partial_{i} \Delta_{\sigma}
$.

\item $\cB_\sigma = u_\sigma *\cB_{\partial \sigma}$ if $\dim\Delta_{\sigma} > 0$.
\item  $\diam \, \cB_\sigma \leq  2\ell / (2\ell +1)$.
\item The simplexes $\{T, \ \tc_{u_\sigma} B \}$ are in strong general position in $\Delta_{\sigma}$ for any $B \in \cB_{\partial \sigma}$ and $T \in \cT_{\sigma}$.
\setcounter{saveenum}{\value{enumi}}
\end{enumerate}
\end{definition}

Observe that Property (PB2) allows the use of the triangulation $\cB_{\partial \sigma}$ of $\partial \Delta$ without
ambiguity. Property (PB1) gives $\cB_{\sigma}=\{u_{\sigma}\}$ if $\dim\sigma=0$.
Before proving the existence of a pseudo-barycentric system we need the following lemma which gives a control in the case of $\ov{p}$-allowable simplexes.

\begin{lemma}\label{pointbase}
Let $(X,\ov p)$ be a perverse space, with a locally finite stratification,
and $\sigma \colon \Delta \to X$  be a $\ov p$-allowable simplex.
Then, the following property, defined for each point  $u\in \inte \Delta$ by
$$
\Theta(u) = 
\hbox{``the simplex $\sigma_u \colon \Delta^0 =[a_0] \to X$, $a_0\mapsto u$, is $\ov p$-allowable,''}
$$
is generic in the sense of \defref{def:genericstable}.
\end{lemma}
\begin{proof} 
The simplex $\sigma_u$ is $\ov p$-allowable if
$
\dim \sigma_u^{-1}S \leq 0  -D \ov p(S) -2
$
for each singular stratum $S \in \cS_X$. 
Since $\dim \sigma_u^{-1}S \leq 0$, this condition  is fulfilled if $D\ov p(S)  \leq -2 $.
It remains to prove that the subset 
\begin{equation}\label{equa:thetau}
\Delta_\Theta =
\left\{u\in\inte \Delta\mid \theta(u) \text{ is true}\right\}
= \inte \Delta \menos \displaystyle \bigcup_{D\ov p(S)  \geq -1} \sigma_{u}^{-1}S
\end{equation}
 is dense in $\Delta$.
 The stratification being locally finite, the previous  union is finite. 
 By dimensional reasons, it suffices to prove that
$\dim \sigma_{u}^{-1}S   \leq \dim \Delta -1$ for each singular stratum $S$ with $D\ov p(S)  \geq  -1$.
The allowability condition for the simplex $\sigma_{u}$ gives 
$$
\dim \sigma_{u}^{-1}S \leq \dim \Delta - D\ov p(S)-2 \leq \dim \Delta -1
$$
and we get the claim.
\end{proof}

%%%%%%%%%%%%%%%
\subsection{Existence of pseudo-barycentric subdivisions}\label{subsec:EPBS}

\begin{proposition}\label{prop:ExisB}
Let $X$ be a filtered  space with a locally finite stratification and let $\cT$  be a simplicial system. 
Then there exists a $\cT$-pseudo-barycentric system $\cB$ of $X$, 
 of associated set of pseudo-barycentres $\cP = \{u_\sigma \in \inte\Delta_{\sigma} \mid \sigma \in \sing\,X\}$. 
Moreover, given a perversity $\ov p$ on $X$, the $\cT$-pseudo-barycentric system $\cB$ can be chosen such that,
for any $\ov{p}$-allowable simplex $\sigma$, the restriction of $\sigma$ to any $B\in \cB_{\sigma}$ with $u_{\sigma}\in B$ is also $\ov{p}$-allowable.
\end{proposition}

\begin{proof}
Let  $\sigma \colon  \Delta^\ell \to X $, we construct $\cB_\sigma$  by induction on $\ell$.
If $\ell=0$ then $\cB_\sigma$ is given by  (PB1). 
For the inductive step, we suppose constructed $\cB_{\partial \sigma}$ verifying properties (PB1)-(PB5) 
and  we have to find a point  
$u_{\sigma}\in \inte \Delta^{\ell}$
such that the triangulation $u_\sigma * \cB_{\partial \sigma}$ verifies the properties  (PB4) and (PB5). 
Let us begin with (PB4). 
We first show that   it suffices to take $u_\sigma$   in the open ball $B\left(b_\sigma, \frac \ell {(\ell+1) (2\ell+1)} \right)$, where $b_\sigma$ is the barycentre of  $\Delta^\ell$.
By induction, we have
$$
\diam \ \cB_{\partial \sigma} \leq
 \frac {2\ell-2} {2\ell -1} \leq   \frac{2\ell }{2\ell +1}.
 $$
Moreover, since $\cB_\sigma  = u_\sigma * \cB_{\partial \sigma}$, it suffices to verify $\dist(u_\sigma,a_j) \leq  \frac{2\ell }{2\ell +1}$
 for each vertex $a_j $ of $\Delta^\ell$. This is a consequence of
  $$
  \dist(u_\sigma,a_j) \leq \dist(u_\sigma,b_\sigma)+\dist(b_\sigma,a_j) 
<
 \frac \ell {(\ell+1) (2\ell+1)} + \frac \ell {\ell+1}  =  \frac {2\ell} {2\ell+1}.
 $$
 Thus,  the properties  (PB4) and (PB5) are satisfied if  
 $$
 B\left(b_\sigma, \frac \ell {(\ell+1) (2\ell+1)} \right)
 \cap
 \bigcap_{\substack{B \in \cB_{\partial \sigma}\\ T \in \cT_{\sigma}} }  \Delta_{\crP(-,T,B)}\ne \emptyset,
 $$
 since this subset is included in $\inte \Delta^\ell$.
 (Recall that $\crP(u,T,B)$ is introduced in \defref{def:strongGP} and $\Delta_{\crP(u,T,B)}$ in  \defref{def:genericstable}.)
 The subdivision $\cB_{\partial \sigma}$ is finite. Since the stratification is locally finite then the family $\cT_\sigma$ is also finite.
 Following \propref{prop:genericstable}, this intersection is a non-empty open subset and it  suffices to take $u_\sigma$ in it.
This gives the first part of the proof.

\medskip
As for the second part, we now suppose that $\sigma$ is $\ov{p}$-allowable.
We choose $u_\sigma$ in the subset
 \begin{equation}\label{base}
 \Delta_\Theta
 \cap
 B\left(b_\sigma, \frac \ell {(\ell+1) (2\ell+1)} \right)
 \cap
 \bigcap_{\substack{B \in \cB_{\partial \sigma}\\ T \in \cT_{\sigma}} }  \Delta_{\crP(-,T,B)}
\end{equation}
where $\Delta_{\Theta}$ is defined in \eqref{equa:thetau}. 
This subset is not empty according to  \lemref{pointbase} and is included in  $\inte \Delta^\ell$.
Now, it suffices to prove that the restriction $\sigma_B$ is a $\ov p$-allowable simplex, that is,
\begin{equation}\label{permisible1}
\dim (B\cap \sigma^{-1}S) = \dim  \sigma_B^{-1} S \leq \dim B - D\ov p(S)-2,
\end{equation}
for each singular stratum $S \in \cS_X$.
If $B =[u_\sigma]$, the result comes from $u_\sigma \in \Delta_\Theta$.
Let us suppose $B = \tc_{u_\sigma} V$ with $V \in \cB_{\partial \sigma}$, the result comes from
 (PB5) and  \propref{prop:PosGen}.
\end{proof}

Let $\sigma\colon \Delta\to X$ be a singular simplex, 
 $F\subset \partial\Delta$ a face and  $u\in\inte \Delta$ a point.
 We define
 \begin{equation}\label{equa:ast}
 u\ast \sigma_{F}\colon \tc_{u}F=u\ast F\to X
 \end{equation}
 as the restriction of $\sigma$ to the cone.
From \propref{prop:ExisB}, we deduce a subdivision  adapted to $\ov{p}$-allowable simplexes.

\begin{proposition} \label{prop:sd}
Let $(X,\ov p)$ be a perverse filtered space, with a locally finite stratification, and \,$\cU$ be an open covering of $X$.
Then there is a chain map, $\sd\colon C_{*}(X;G)\to C_{*}(X;G)$, satisfying the following properties.
\begin{enumerate}[1)]
\item The image of a $\ov{p}$-allowable chain is a $\ov{p}$-allowable chain.
\item For each simplex $\sigma\colon\Delta\to X$, there exists $r\in\N$ such that the geometric support of $\sd^r \sigma$ is included in an element of $\cU$.
\end{enumerate}
\end{proposition}

\begin{proof}
 Let $\cT$ be a simplicial system of $X$ and $\cB$ be a $\cT$-pseudo-barycentric subdivision of $X$ as in \propref{prop:ExisB}.
 In degree 0, we define $\sd$ as the identity map. 
 Suppose that we have constructed, for some $\ell \in \N$, a morphism $\sd\colon C_{*< \ell} (X) \to C_{*<\ell} ( X)$ verifying the two following properties.
 \begin{itemize}
 \item[(i)]  The morphism $\sd$ is a chain map, i.e. $\partial \circ \sd = \sd  \circ \partial$.
 
\item[(ii)] If $\tau\colon \Delta^k\to X$, with $k<\ell$, there exists a finite family $\{F_i\}_{i\in I}  \subset  \cB_\tau$ with 
$ \sd \,\tau =\sum_{i\in I} n_i \tau_{F_i} , n_i \in G$. Moreover, if $\tau$ is $\ov{p}$-allowable, then so is $\sd \,\tau$.
\end{itemize}
We now prove (i) and (ii) for a   simplex $ \sigma \colon \Delta^\ell \to X$. 
By induction, there is a family  $\{F_i\}_{i\in I}  \subset  \cB_{\partial \sigma}$ with 
$ \sd (\partial \sigma)  = \sum_{i\in I} n_i \,\sigma_{F_i} $, $n_i \in G$.
 If $u_{\sigma}$ is a pseudo-barycentre, 
from  Property (PB3) of \defref{defB}, we set 
$$
\sd \,\sigma =u_\sigma * \sd (\partial \sigma)  \coloneqq  \sum_{i\in I} n_i \, u_\sigma * \sigma_{F_i},
$$
where the last equality comes from the induction
and $u_{\sigma}\ast\sigma_{F_{i}}$ is defined in \eqref{equa:ast}.
 \propref{prop:ExisB} implies that $\sd\, \sigma$ is $\ov{p}$-allowable if $\sigma$ is so. We have established Property (ii). 
 Finally, Property (i) comes from the induction and
\begin{eqnarray*}
\partial (\sd\, \sigma) 
&=&
 \sum_{i\in I} n_i \,\partial (u_\sigma * \sigma_{F_{i}})=  \sum_{i\in I} n_i \,\sigma_{F_i}  
-
u_{\sigma}\ast \partial\left(\sum_{i\in I} n_i \, \sigma_{F_{i}}\right)\\
&=&
 \sd(\partial \sigma) - u_\sigma *\partial (\sd(\partial \sigma)) 
 =
  \sd(\partial \sigma).
\end{eqnarray*}
Property 2) of the statement comes from  (PB4) with a classical Lebesgue number argument.
\end{proof}

%%%%%%%%%%%%%%%%%%%%%%%%%
\subsection{Homotopy of pseudo-barycentric subdivision}\label{subsec:PBScylinder}
We begin with an adaptation of  \propref{defB} in order to construct
a homotopy operator in \propref{prop:sdT}.

\begin{proposition}\label{613}
 Let $X$  be a filtered space, with a locally finite stratification, and $\cT$  be a simplicial system.
 Let $\cB$ be a $\cT$-pseudo-barycentric subdivision. 
 Then, for each simplex $\sigma \colon \Delta \to X$ 
 there exists a triangulation $\widetilde{\cB_\sigma}$  of $\Delta \times [0,1]$ verifying the following properties:
\begin{enumerate}[\rm (PB1)]
\setcounter{enumi}{\value{saveenum}}
\item  $\widetilde{\cB_\sigma}  =\{\Delta \times [0,1]\}$ if $\dim \sigma =0$,
\item  $( \widetilde{\cB_\sigma} )_{\partial_{i}\sigma} = \widetilde{\cB_{\partial_{i}\sigma}}$, for any $i$-face of $\sigma$,
 \item $\widetilde{\cB_\sigma} =
 \left( (u_\sigma,1)* \widetilde{\cB_{\partial \sigma}} \right)
 \cup
 \left( (u_\sigma,1)* (\Delta \times \{ 0\}\right)$,
 \\
 with
 $( u_\sigma,1)* \widetilde{\cB_{\partial \sigma}}  = 
 \left\{ 
 F \in  \widetilde{\cB_{\partial \sigma}}
 \right\} 
 \cup 
  \left\{ (u_\sigma,1) * F \mid F \in  \widetilde{\cB_{\partial \sigma}} 
  \right\},
 $
\item $\pr(\widetilde{\cB_\sigma} ) = \cB_\sigma \cup \{ \Delta\} $, where $\pr\colon  \Delta \times [0,1] \to \Delta$  is the canonical projection.
\end{enumerate}
 \end{proposition}
 
 Observe that  (PB7) allows the use of the triangulation
 $\widetilde{\cB_{\partial \sigma}}$  of $\partial \Delta$   and justifies (PB8).
 
\begin{proof}
The construction of $\cB_\sigma$ verifying (PB6)-(PB8) is straightforward by induction. It remains to prove (PB9). 
We proceed by induction on the dimension of $ \sigma$. If $\dim \sigma=0$, the result is clear.
 For the
inductive case, we consider $H \in \widetilde{\cB_\sigma}$ and compute $\pr(H)$. There are three cases.
\begin{itemize}
\item If $H \in \widetilde{\cB_{\partial \sigma}}$,  the induction hypothesis implies  $\pr(H) \in \cB_{\partial \sigma} \subset  \cB_\sigma$.
\item If $H = (u_\sigma,1)*  F$ with $F \in  \cB_{\partial \sigma}$, the induction hypothesis implies 
$\pr(F) \in  \cB_{\partial \sigma}$.  Thus, we have $\pr(H)=\pr((u_\sigma,1)*  F)= u_\sigma  * \pr(F) \in u_\sigma  * \cB_{\partial \sigma} \subset \cB_\sigma$.
\item Finally, $H=(u_\sigma,1)*  (\Delta  \times \{0\})$  gives $\pr(H)=\Delta$.
\end{itemize}
\end{proof}

Let us observe the following straightforward point.

\begin{lemma}\label{lem:ontoproj}
Let $\Delta$ be an Euclidean simplex and $\pr \colon \Delta \times [0,1] \to \Delta$ be the canonical projection.
Given a simplex $H\subset \Delta \times [0,1]$ and a subset $A \subset \pr\,H$, we have
$$
\dim (\pr^{-1}(A )\cap H) \leq \dim A + \dim H - \dim \pr \,H\,.
$$
\end{lemma}

We now construct a homotopy between the chain map $\sd$ of \propref{prop:sd} and the identity. 
 
\begin{proposition}\label{prop:sdT}
Let $(X,\ov p)$ be a perverse filtered space, with a locally finite stratification.
Then there exists a morphism $T \colon  C_*(X; G) \to  C_{*+1}(X; G)$ 
verifying $\id -  \sd = T \partial + \partial T$ and such that
the image by $T$ of a $\ov{p}$-allowable chain is a $\ov{p}$-allowable chain.
\end{proposition}

\begin{proof}
 Let $\cT$ be a simplicial system of $X$ and $\cB$ be a $\cT$-pseudo-barycentric subdivision as in \propref{prop:ExisB}.
 Finally, let $\widetilde{\cB_\sigma}$ be a family given by \propref{613}, where  $\sigma \colon\Delta\to X$. 
 Suppose that we have constructed, for some $\ell \in \N$, a morphism $T \colon C_{*< \ell} ( X) \to  C_{*+1<\ell+1}( X)$ verifying the two following properties.
\begin{itemize}
\item[(iii)] $\id-\sd=T\partial +\partial T$.
\item[(iv)] If $\tau \in  C_{*< \ell} ( X)$, there exists a family $\{F_i\}_{i\in I} \subset \widetilde{ \cB_\tau}$ with
$T(\tau)= \sum_{i\in I} n_i (\tau \circ \pr)_{F_i}$.
Moreover, if $\tau$ is $\ov{p}$-allowable, then so is $T(\tau)$.
\end{itemize}
Let $\sigma \colon \Delta^\ell \to X$. %
By induction, there is a family $\{F_i\}_{i \in I} \subset \widetilde{ \cB_{\partial \sigma}}$ 
with $T(\partial \sigma) =\sum_{i\in I}n_i (\sigma \circ \pr)_{F_i}.$
For each $F_{i}$, $i \in I$, we use the notation previously introduced in \eqref{equa:ast}. More precisely, we set
 \begin{itemize}
 \item  $u_\sigma * (\sigma \circ \pr)_{F_i}\colon 
 \tc_{(u_{\sigma},1)} F_{i}=(u_\sigma,1) * F_i \to X$ 
  for  the restriction of $\sigma \circ \pr \colon \Delta\times [0,1]\to X$.
\item  In the particular case where $F_{i}=\Delta\times \{0\}$, we simplify the notation 
and write $u_{\sigma}\ast \sigma$ instead of $(u_\sigma, 1) *  (\Delta\times \{0\})$. 
\end{itemize}
\noindent We set
\begin{equation}\label{defT}
T(\sigma )=u_\sigma * (\sigma  - T(\partial \sigma))
\coloneqq
u_\sigma *\sigma - \sum_{i\in I} n_i  (u_\sigma * (\sigma \circ \pr)_{F_i}).
\end{equation}
We claim that $T(\sigma)$ is $\ov p$-allowable if $\sigma$ is so,  which implies Property (iv).
We have two cases.

\smallskip
(I) Set $H=(u_\sigma,1)*(\Delta\times \{0\})$. Since the map  $\pr$ is the canonical projection and $  \pr(H) = \Delta$, 
we can apply  \lemref{lem:ontoproj} and  get, for each stratum $S\in \cS_X$:
\begin{eqnarray*}
\dim((\sigma \circ \pr)^{-1} (S)\cap  H)& =&  
\dim( \pr^{-1} \sigma^{-1} (S)\cap  H)  \leq   \dim( \sigma^{-1} S) +  \dim H - \dim \Delta
\\
&\leq_{(1)}& \dim \Delta - D\ov p(S)-2+\dim H - \dim \Delta \\
&=& \dim H- D\ov p(S)-2,
\end{eqnarray*}
where the inequality (1) comes from \eqref{def:homotopygeom}.

\smallskip
(II) Set $H=(u_\sigma,1)*F_i$.
From Property (PB9) of \propref{613}, we have $\pr (H) \in \cB_\sigma$ and we get: 
\begin{eqnarray*}
\dim(\pr^{-1} (T )\cap  H) 
&\leq_{(2)} &\dim (T  \cap \pr (H)) +\dim H - \dim \pr(H) 
\\
&\leq_{(3)} &   \dim T  +  \dim \pr(H) -\dim \Delta + \dim H  -  \dim \pr(H) 
\\
&\leq_{(4)}& \dim \sigma^{-1}(S)-\dim \Delta  +  \dim H  \\
&\leq _{(5)}&
\dim \Delta - D\ov p (S)-2 -\dim \Delta  +  \dim H 
\\
&=&
 \dim H  - D\ov p (S)-2,
\end{eqnarray*}
where $(2)$ comes from \lemref{lem:ontoproj},  (3) from (PB5), (4) from \defref{def:envl} and (5) from the inequality \eqref{equa:admissible}.

\medskip
We have established Property (iv). Finally, Property (iii) is a consequence of the induction and the construction of $\sd$ made in \propref{prop:sd},
\begin{eqnarray*}
\partial T(\sigma) &=&
\partial (u_\sigma *(\sigma- T(\partial \sigma)))
=\sigma-T(\partial \sigma)-u_\sigma *\partial (\sigma-T(\partial \sigma))
\\
&=& \sigma- T(\partial \sigma)-u_\sigma * \sd(\partial \sigma)=\sigma - T(\partial \sigma)-\sd(\sigma).
 \end{eqnarray*}
 \end{proof}

%%%%%%%%%%%%%%%%%%%%%
\section{Mayer-Vietoris sequence}\label{sec:MV}

In this section, we prove the following result.

\begin{theoremb}\label{thm:MVnewdim}
Let $(X,\ov{p})$ be a perverse space and $\{U,V\}$ be an open covering of $X$. Then there is a long exact sequence,
called Mayer-Vietoris sequence,
$$
\ldots \to
H_{*}^{\ov{p}} (U \cap V) \to  H_{*}^{\ov{p}}(U) \oplus  H_{*}^{\ov{p}} (V)  \to  H_{*}^{\ov{p}}(X)   \to  H_{*-1}^{\ov{p}}(U \cap V)  \to  \dots 
$$
\end{theoremb}

Let us consider an intersection chain $\xi$ made up of a simplex. 
Then there exists an integer $r\in \N$ such that the simplexes of $\sd^r \xi$ belongs to 
$C_{*}^{\ov{p}}(U)$ or $C_{*}^{\ov{p}}(V)$ (cf. \propref{prop:sd}). 
The same property happens if all the simplexes of $\xi$ are intersection simplexes. 
So we still have the case of allowable simplexes that are not intersection chains. 
To study them, we follow a similar method used in \cite[Proposition A.14.(i)]{CST1}. 
The key point is that the allowability defect of the boundary of an allowable simplex is concentrated in only one face,
cf \propref{bad-1}. 

Let us notice that a  $\ov p$-allowable simplex $\sigma \colon \Delta \to X$  is not a $\ov p$-intersection chain 
if, and only if,
 there exists a codimension one face $\nabla$ of $ \Delta$ and a singular stratum $S \in \cS_X$   with
$$
0\leq \dim (\nabla  \cap \sigma^{-1}S)  = \dim \Delta -D \ov p(S) -2.
$$
To study them, we introduce the following definition.

\begin{definition}
Let $(X,\ov p)$ be a perverse space and  $\sigma \colon \Delta \to X$ be a $\ov p$-allowable simplex.
A face $F\triangleleft \Delta$ is a \emph{critical face}
if $F\ne \Delta$ and there exists a singular stratum $S\in  \cS_X$ verifying
$$
0\leq \dim (F \cap \sigma^{-1}S) = \dim \Delta -D \ov p(S) -2.
$$
The set of critical faces of $\sigma$ is denoted by $\cC_\sigma$.
\end{definition}

The family of critical faces of a pseudo-barycentric subdivision has a nice structure that we present now.
To do it, we first introduce the following definition.

\begin{definition}\label{def:Complet}
Let $\sigma \colon \Delta \to X$ be a singular simplex and $\cB_\sigma$ be a pseudo-barycentric subdivision of $\Delta$.
If $\{u_0, \dots,u_\ell\}$ are the vertices of $\Delta$, the family of pseudo-barycentres of $\cB_\sigma$ is of the shape 
$$
\cP_\sigma = \{ u_{i_0 \dots i_a} \mid  0 \leq i_0 < \dots < i_a \leq  \ell\},
$$
where $u_{i_0 \dots i_a}$ is the pseudo-barycentre contained in the interior of the face $[u_{i_0}, \dots,u_{i_a}]$.

A face $B \in \cB_\sigma$ such that $\dim B = \dim \sigma$
is of the type
$B = [u_{j_0}, u_{j_0j_1} , \dots , u_{j_0 \dots j_\ell}] $ with $\{j_0, \dots, j_\ell \} =  \{i_0,\dots,i_\ell\}$.
The faces $
F = [u_{j_0}, u_{j_0j_1} , \dots , u_{j_0 \dots j_a}]$, with $a \in \{0, \dots,\ell\}$, 
are  called \emph{complete faces of $B$}.
\end{definition}

\begin{lemma}\label{minimal}
 Let $(X,\ov p)$  be a perverse space with a locally finite stratification, $\cT$  be a simplicial system and
 $\cB$ be a $\cT$-pseudo-barycentric subdivision given by \propref{prop:ExisB}.
For any $\ov p$-allowable simplex $\sigma \colon \Delta \to X$ and
for any $B \in \cB_\sigma$ with $\dim B = \dim \Delta$, 
the set  of critical faces, $\cC_{\sigma_B}$, is empty or  
has a minimum element  $M$; i.e., for every $F\in \cC_{\sigma_B}$, one has
$M\subseteq F$. 
If it exists, this minimal element $M \triangleleft \Delta$ is called  the \emph{$\ov p$-bad face} of $B$. 
By extension, the singular simplex $\sigma_M\colon M\to X$ is also called the \emph{$\ov p$-bad face} of $\sigma_B$.
\end{lemma}

\begin{proof}
We prove that a minimal element $F$ of $\cC_{\sigma_B}$ is complete. This gives the claim since the complete faces of $B$ are totally ordered by their dimension. 

$\bullet$ Let us begin with some calculations. 
A critical face $F$ of $B$ relatively to the singular stratum $S\in \cS_X$, verifies
\begin{equation}\label{bat}
 \dim (F \cap \sigma_B^{-1}S) = \dim B -D \ov p(S) -2  \geq 0.
\end{equation}
Since the simplex $\sigma$ is $\ov p$-allowable, we also have 
$$
\dim B -D \ov p(S) -2 =  \dim (F \cap \sigma_B^{-1}S) \leq  \dim \sigma_B^{-1}S \leq 
\dim \sigma^{-1}S \leq
\dim \Delta -D \ov p(S) -2,
$$ 
and therefore
\begin{equation}\label{bi}
\dim \sigma_B^{-1}S  = \dim B -D \ov p(S) -2.
\end{equation}
Using the simplicial envelope of \defref{def:envl}, we get $\displaystyle \sigma_B^{-1}S \subset \cup_{T \in \cT_{\sigma_B,S} }T$ and
\begin{equation}\label{hiru}
\dim  \sigma_B^{-1}S = \max \{\dim T \mid T \in \cT_{\sigma_B,S} \}.
\end{equation}
With the equalities \eqref{bi}, \eqref{hiru}, \eqref{bat},  the inclusion
$\displaystyle  F \cap \sigma_B^{-1}S \subset \cup_{T \in \cT_{\sigma_B,S} }F \cap T$ 
implies
\begin{eqnarray*}
\dim B -D \ov p(S) -2 
&=&
  \max \{\dim T \mid T \in \cT_{\sigma_B,S} \} \\
&\geq&
  \max\{ \dim (F \cap T) \mid T \in \cT_{\sigma_B,S} \} \\
 & =&
  \dim (  \cup_{T \in \cT_{\sigma_B,S} }F \cap T)\\
   &\geq& 
\dim ( F \cap \sigma_B^{-1}S )  \\
&=& 
\dim B -D \ov p(S) -2 .
\end{eqnarray*}
We conclude:
$
   \max\{ \dim (F \cap T) \mid T \in \cT_{\sigma_B,S} \}   =  \dim B -D \ov p(S) -2 
 $.
In the sequel, we use the decomposition 
$$
 \cT_{\sigma_B,S} = \cT_1 \cup \cT_2,
 $$
 where
$$
\cT_1 = \{ T \in  \cT_{\sigma_B,S}  \mid \dim (F\cap T) = \dim \Delta - D\ov p(S) -2\}
$$
and
$$
\cT_2 = \{ T \in   \cT_{\sigma_B,S}  \mid \dim (F\cap T) <  \dim \Delta - D\ov p(S) -2\}.
$$

\medskip
$\bullet$ We now prove  that the face $F$ is complete. 
Without loss of generality, we can assume $B =[u_0,u_{01}, \dots,u_{01\dots\ell}]$.
Let $a\in \{0,\dots,\ell\}$ be the smallest number with $F  \subset [u_0,u_{01}, \dots,u_{01\dots a}] =E$.
In particular $F =\tc_{u_{01\dots a}} V$ where $V\triangleleft [u_0,u_{01}, \dots,u_{01\dots {(a-1)}}]$ or $V =\emptyset$.
We distinguish these two cases. 

\medskip 
1) $V=\emptyset$. Then $F= [u_{01\dots a}]$.
Recall from \lemref{pointbase} and \eqref{base} that $\sigma_F$ is $\ov p$-allowable.
Using \eqref{bat} and this allowability, we get
$$
\dim B -D \ov p(S) -2 =  \dim (F \cap \sigma^{-1}_B S)   = \dim \sigma^{-1}_F S \leq 0 -D \ov p(S) -2,
$$
which gives
$\dim B=0$ and therefore $a =\ell =0$. 
The face  $F = [u_{0}]$ is complete.

\smallskip
2) $V\ne \emptyset$.
Then the  simplexes  $\{F =\tc_{u_{01\dots a}} V,T\}$ are in strong general position in $E$ for
all $T\in \cT_{\sigma}$  (cf. (PB5)). 
Inspired by \defref{def:strongGP}, we distinguish   two possibilities.

\medskip
i) There exists $T \in \cT_1$ with  $F \cap T \not\subset  \partial E$. We have 
 \begin{eqnarray*}
 \dim \Delta - D\ov p(S)  - 2
 & =&
 \dim ( F \cap T) =_{(1)} \dim F + \dim T - \dim E \\
 &\leq_{(2)} & 
 \dim F + \dim \sigma_B^{-1}S  - \dim E \\
 &\leq&
  \dim F+   \dim \sigma^{-1}S - \dim E \\
&\leq_{(3)}   & 
\dim F + \dim \Delta - D\ov p(S)  - 2 - \dim E,
\end{eqnarray*}
where $(1)$ is \eqref{SGP}, $(2)$ is \eqref{hiru} and $(3)$ comes from the $\ov{p}$-allowability of $\sigma$.
This implies $\dim E \leq \dim F$ and therefore $F=E$ which is a complete face.

\medskip
ii) Or, we have the inclusion $F \cap T \subset \partial E$,  for each  $T \in \cT_1$. This implies,  
\begin{equation}\label{bost}
F \cap T = V \cap T \subset V.
\end{equation}

The next step is the determination of  $\dim  (V \cap \sigma_B^{-1}S)$.
We set $Z_{1}=\cup_{T\in\cT_{1}}(V\cap T)$ and $Z_{2}=\cup_{T\in\cT_{2}}(V\cap T)$.
From the decomposition $V\cap \sigma_{B}^{-1}S\subset Z_{1}\cup Z_{2}$, 
we have
\begin{equation}\label{zortzi}
\dim (V \cap \sigma_B^{-1}S )  = 
\max\{ \dim (V \cap \sigma_B^{-1}S \cap Z_1) , \dim (V \cap \sigma_B^{-1}S  \cap Z_2) \}.
\end{equation}
We compute these two terms:
\begin{eqnarray}\label{bederatzi}
 \dim (V \cap \sigma_B^{-1}S  \cap Z_2) 
& =&
 \max\{ \dim(V \cap \sigma_B^{-1}S   \cap T )\mid T \in \cT_2
 \}
 \\ \nonumber
 &\leq &
  \max\{ \dim(F   \cap T )\mid  T \in \cT_2
 \}
<
 \dim \Delta - D\ov p(S) -2,
\\[,2cm]\label{sei}
 \dim (V \cap \sigma_B^{-1}S  \cap Z_1) 
& = &
 \max\{ \dim(V \cap \sigma_B^{-1}S  \cap T\mid T \in \cT_1
 \}
 \\
  &=_{\eqref{bost}}&
  \max\{ \dim(F \cap \sigma_B^{-1}S   \cap T )\mid T \in \cT_1
 \}.
 \end{eqnarray}
 For the determination of this last term, we set
 $W_{1}=\cup_{T\in \cT_{1}}(F\cap T)$, $W_{2}=\cup_{T\in \cT_{2}}(F\cap T)$
 and decompose $F\cap \sigma_{B}^{-1}S\subset W_{1}\cup W_{2}$.
We obtain
\begin{eqnarray*}
  \dim B - D\ov p(S) -2  
  &=_{\eqref{bat}}&
   \dim (F \cap \sigma_B^{-1}S )\\
   &  =&
\max\{ \dim (F \cap \sigma_B^{-1}S \cap W_1) , \dim (F \cap \sigma_B^{-1}S  \cap W_2) \}.
\end{eqnarray*}
From 
\begin{eqnarray*}
 \dim (F \cap \sigma_B^{-1}S  \cap W_2) 
& =&
 \max\{ \dim(F \cap \sigma_B^{-1}S   \cap T )\mid  T \in \cT_2
 \}
 \\
 &\leq &
  \max\{ \dim(F   \cap T )\mid T \in \cT_2
 \}
<
 \dim \Delta - D\ov p(S) -2
\end{eqnarray*}
and $\dim \Delta=\dim B$, we conclude 
\begin{eqnarray}\label{zazpi} \label{hamar}
  \dim B - D\ov p(S) -2 & =  & \dim (F \cap \sigma_B^{-1}S  \cap W_1)  \\ \nonumber 
& =&
    \max\{ \dim(F \cap \sigma_B^{-1}S   \cap T )\mid T \in \cT_1
 \}\\ \nonumber 
 &=_{\eqref{sei}}&
  \dim (V \cap \sigma_B^{-1}S  \cap Z_1). 
\end{eqnarray}
Finally, from \eqref{zortzi}, \eqref{bederatzi} and \eqref{hamar}
 we get that 
 $
   \dim (V \cap \sigma_B^{-1}S )  =   \dim B - D\ov p(S) -2$.
   In short,  the face $V$ is a critical face of $B$ with $V \triangleleft F$. 
   This is impossible since $F$ is minimal and $\dim V < \dim F$.
\end{proof}

\begin{proposition}\label{bad-1}
 Let $(X,\ov p)$  be a perverse space with a locally finite stratification, $\cT$  be a simplicial system and
 $\cB$ be a $\cT$-pseudo-barycentric subdivision whose existence is guaranteed by \propref{prop:ExisB}.
Consider a $\ov p$-allowable simplex $\sigma \colon \Delta \to X$. 
Then, for any $B \in \cB_\sigma$ with $\dim B = \dim \Delta$, 
we have the following properties.
\begin{enumerate}[\rm (a)]

\item  Let $\tau \colon \nabla \to X$ be a face of  codimension one of $\sigma_B$. 
The simplex $\tau$ is not  $\ov{p}$-allowable if, and only if,
$B$ has a $\ov{p}$-bad face  included in $\nabla$.
\item Let  $B'$ be another  simplex of $\cB_\sigma$ with $\dim B'=\dim \Delta$.
If $\sigma_B$ and  $\sigma_{B'}$ share  a codimension one  face  $\tau$ which is  not $\ov p$-allowable, then 
$\sigma_B$ and $\sigma_{B'}$ have the same  $\ov{p}$-bad face. Moreover, this face is a face of $\tau$.
\item The simplex  $\sigma_B$ is a $\ov{p}$-intersection chain if, and only if, it has no $\ov{p}$-bad face.
\end{enumerate}
\end{proposition}

\begin{proof} Recall from \propref{prop:ExisB} that the simplex $\sigma_B$ is $\ov p$-allowable and verifies \eqref{equa:admissible}.

(a) Let $\tau \colon \nabla \to X$ be a face of  codimension one of $\sigma_B$. 

$\bullet$ Let us suppose that the face  $\tau$ is not  $\ov{p}$-allowable. 
So, there exists a stratum  $S \in \cS_{X}$ such that  $\dim \sigma^{-1}S\geq 0$ and
$$
\dim \sigma^{-1}_B S   \geq \dim \tau ^{-1} S  \geq 1+\dim \nabla - D \ov{p}(S) -2
=\dim \Delta-D\ov{p}(S)-2 \geq \dim \sigma^{-1}_B S.
$$
We deduce $\dim (\nabla \cap \sigma_B^{-1}S) =  \dim \sigma ^{-1} S = \dim \Delta - D \ov{p}(S) -2 \geq 0$ 
and  $\nabla \in \cC_B$. 
So, it exists a $\ov p$-bad face $M$ of $B$ with $M\triangleleft \nabla$ (cf. \lemref{minimal}).

$\bullet$ Conversely, let us suppose that the $\ov{p}$-bad face $M$ of  $B$ exists and verifies 
$M \subset \nabla$. 
There is therefore a singular stratum $S \in \cS_{X}$ with
$
\dim (M \cap \sigma_B^{-1}S) = \dim B- D\ov{p}(S)-2\geq 0.
$
So, using also that $\sigma$ is $\ov{p}$-allowable, we get
\begin{eqnarray*}
 \dim \Delta- D\ov{p}(S) -2
 & \geq &
 \dim \sigma^{-1} S 
 \geq  
 \dim (\nabla \cap \sigma_B^{-1}S)\\
 &\geq& \dim (M \cap \sigma_B^{-1}S) 
= \dim B- D \ov{p}(S) -2.
\end{eqnarray*}
This implies
 $$  \dim (\nabla \cap \sigma_B^{-1}S)  = \dim B- D\ov{p}(S)-2 > \dim \nabla- D\ov{p}(S)-2$$
and we conclude that $\sigma_{\nabla}\colon \nabla \to X$ is not $\ov p$-allowable.

\smallskip
(b) 
Let $M ,M'$ be the minimal faces of $\cC_B$ and $\cC_{B'}$ respectively. 
Following (a) we have  that $\sigma_M,\sigma_{M'}$ are faces of $\tau$ and thus of $\sigma_B$.  
By minimality $M=M'$ and the  bad faces of $\sigma_B$ and $\sigma_{B'}$ are the same.
Moreover this face is a face of $\tau$.

\smallskip
(c) The simplex $\sigma_B$ is an intersection chain if, and only if, its codimension one faces are $\ov p$-allowable.
Following (a), this is equivalent to the non existence of $\ov p$-bad faces.
\end{proof}

\begin{proof}[Proof of \thmref{thm:MVnewdim}]
 Let us consider the short exact sequence
\begin{equation}\label{equa:MVC}
\xymatrix@C=5mm{	
0 \ar[r] & C^{\ov{p}}_{*}(U\cap V)
\ar[r] & 
C^{\ov{p}}_{*}(U)\oplus C^{\ov{p}}_{*}(V)
\ar[r]^{\varphi} & 
 C^{\ov{p}}_{*}(U)+ C^{\ov{p}}_{*}(V)
 \ar[r] & 
 0\\
}\end{equation}
where the chain map  $\varphi$ is defined by  $(\alpha,\beta) \mapsto \alpha + \beta$.
The existence of the Mayer-Vietoris exact sequence comes from the fact
that the inclusion $\im \varphi \hookrightarrow C^{\ov{p}}_{*}(X)$
induces an  isomorphism in homology.
We decompose this proof in two steps.

\smallskip
$\bullet$ \emph{First step: 
The subdivision operator is homotopic to the identity}.  
In Propositions \ref{prop:sd} and \ref{prop:sdT}, we have constructed  two operators preserving the $\ov{p}$-allowability
$\sd \colon C_{*}(X)\to C_{*}(X)$  and
$T\colon C_{*}(X) \to C_{*+1}(X)$, verifying $\sd \circ \partial = \partial  \circ \ \sd$ and 
$\partial T + T \partial  = \id -\sd$. 
Therefore, 
$\sd \colon C_{*}^{\ov p} (X)\to C_{*}^{\ov p} (X)$  and
$T\colon C_{*}^{\ov p} (X) \to C_{*+1}^{\ov p} (X)$ 
are well defined.

\smallskip

$\bullet$ \emph{Second step: We have the following implication:
\begin{equation}\label{equa:subdivisionetintersection}
\xi  \in C^{\ov{p}}_{*}(X)\Rightarrow  \sd^k \xi \in 
C^{\ov{p}}_{*}(U)+ C^{\ov{p}}_{*}(V) \text{ for some } k \geq 1.
\end{equation}
}
Let $\xi$ be a $\ov p$-allowable chain of $X$. The \emph{canonical decomposition} of  $\sd \, \xi$ is $\sd \, \xi = \xi_0 + \sum_{\mu \in I_\xi} \xi_\mu$ where:
\begin{itemize}
\item[-] $\xi_0$ is the chain containing the simplexes of $\xi$ without $\ov p$-bad faces;
\item[-] $I_\xi$ is the family of the $\ov p$-bad faces of the simplexes of $\sd \, \xi$;
\item[-] $\xi_\mu$ is the chain containing the simplexes of $\sd \, \xi$ having $\mu$ as $\ov p$-bad face.
\end{itemize}
The boundary $\partial \xi_0$ is a  $\ov p$-allowable chain.
A non-$\ov p$-allowable face  $\tau$ of a simplex $\sigma$ of $\xi_\mu$ contains necessarily $\mu$. 
When $\partial \xi$, and therefore $ \sd(\partial \xi) = \partial \sd \, \xi$,    is a $\ov p$-allowable chain (cf. \propref{prop:sd}) then $\tau$ does not appear in 
$\partial \sd \, \xi$.
So, there exists another simplex $\sigma'$ of $\sd \, \xi$ having $\tau$ as a face. 
Since $\tau$ contains $\mu$ then $\mu$ is the bad face of $\sigma'$.
We conclude that 
 $\partial \xi_\mu$ is also a $\ov p$-allowable chain.
 (These facts come from \propref{bad-1}.) For each $\mu \in I_{\xi}$, we have proven
 $$
\xi\in C^{\ov{p}}_{*}(X)\Rightarrow  
 \xi_0, \xi_\mu \in C^{\ov{p}}_{*}(X).
 $$
The usual subdivision argument gives the existence of an integer $k\geq 1$ such that 
the canonical decomposition of $\sd^k \xi$ verifies the following properties.
\begin{itemize}
\item[-] Each simplex of $(\sd^k \xi)_0$ lives in $U$ or in $V$.
\item[-] For each $\tau \in I_{\sd^k \xi}$ the chain $(\sd^k \xi)_\tau$ lives in $U$ or in $V$.
\end{itemize}
This gives \eqref{equa:subdivisionetintersection}.
\end{proof}

%%%%%%%%%%
\section{Intersection homology is intersection homology}\label{sec:twointersec}

In this section, we prove that the polyhedral dimension of  \defref{def:envl} brings an intersection
homology isomorphic  to that of King. Thus, this is a ``reasonable dimension''.

\subsection{Original intersection homology}
Let us  specify what we mean by ``original intersection homology.''
We use the expression $\ov{p}$-GM-allowable to make a distinction with  \defref{def:homotopygeom}. 

\begin{definition}\label{def:Kingallowable}
Let $(X,\ov{p})$  be a  perverse space.
A simplex $\sigma\colon \Delta\to X$ is  
\emph{$\ov{p}$-GM-allowable}  %
if, for each stratum $S$, the set $\sigma^{-1}S$ 
is included in the ($\dim\Delta-\codim S +\ov{p}(S)$)-skeleton of $\Delta$.
A singular  chain $\xi$ is \emph{$\ov{p}$-GM-allowable} if it can be written as a linear combination of 
$\ov{p}$-GM-allowable  simplexes,
and of \emph{$\ov{p}$-GM-intersection} if $\xi$ and its boundary $\partial \xi$ are $\ov{p}$-GM-allowable.
We denote by  $I^{\ov{p}}C_{*}(X;G)$ the complex of singular chains of $\ov{p}$-intersection
and  $I^{\ov{p}}H_{*}(X;G)$ its homology, called
\emph{$\ov{p}$-GM-intersection homology of $X$ with coefficients in an $R$-module $G$.}
\end{definition}

In the case of GM-perversities, this definition coincides with that of King.
It uses the more general framework of MacPherson perversities (\cite{MacPherson90} and \defref{def:perversitegen}).
We will refer to \cite{CST3} or \cite{FriedmanBook} for the main related properties: 
intersection homology of a cone, Mayer-Vietoris sequence, ...

 \begin{theoremb}\label{thm:intersecgajer} %
 Let $(X,\ov{p})$ be a perverse CS set.
 The intersection homology $H_*^{\ov{p}}(X;G)$ of \defref{def:homotopygeom}
  is isomorphic to the  $GM$-intersection homology,
 $I^{\ov{p}}H_{*}(X;G)$, introduced by H.C.~King in \cite{MR800845}
 and recalled in \defref{def:Kingallowable}.
\end{theoremb}

The method of  proof is a variant of \cite[Theorem 10]{MR800845}, \cite[Theorem B]{CST3}.

\begin{theoremb}{\rm(\cite[Theorem 5.1.4]{FriedmanBook})}\label{thm:gregtransformationnaturelle}
Let $\cF_{X}$ be the category whose objects are (stratified homeomorphic to) open subsets
of a given  CS set $X$ and whose morphisms are  stratified homeomorphisms and inclusions.
Let  $\cAb_{*}$ be the category of graded abelian groups. Let $F_{*},\,G_{*}\colon \cF_{X}\to \cAb$
be two functors and
 $\Phi\colon F_{*}\to G_{*}$ be a natural transformation satisfying
 the conditions listed
below.
\begin{enumerate}[(i)]
\item $F_{*}$ and $G_{*}$ admit exact Mayer-Vietoris sequences and the natural transformation $\Phi$ 
 induces a commutative diagram between these sequences.
\item If $\{U_{\alpha}\}$ is an increasing collection of open subsets of $X$  and 
$\Phi\colon F_{*}(U_{\alpha})\to G_{*}(U_{\alpha})$ is an isomorphism 
for each $\alpha$, then $\Phi\colon F_{*}(\cup_{\alpha}U_{\alpha})\to G_{*}(\cup_{\alpha}U_{\alpha})$  is an isomorphism.
\item If $L$ is a compact filtered space such that 
$X$  has an open subset  which is stratified homeomorphic
to $\R^i\times \rc L$ and, if
$\Phi\colon F_{*}(\R^i\times (\rc L\backslash \{\tv\}))\to G_{*}(\R^i\times (\rc L\backslash \{\tv\}))$
is an isomorphism, then so is
$\Phi\colon F_{*}(\R^i\times \rc L)\to G_{*}(\R^i\times \rc L)$. Here, $\tv$ is the apex of the cone $\rc L$.

\item If $U$ is an open subset of X contained within a single stratum and homeomorphic
to  an Euclidean space, then $\Phi\colon F_{*}(U)\to G_{*}(U)$ is an isomorphism.
\end{enumerate}
Then $\Phi\colon F_{*}(X)\to G_{*}(X)$ is an isomorphism.
\end{theoremb}

\subsection{Proof of \thmref{thm:intersecgajer}}
We verify the conditions of  \thmref{thm:gregtransformationnaturelle} 
for the natural transformation  
$\Phi\colon I^{\ov{p}}H_{*}(U)\to H_{*}^{\ov p}(U)$
induced by the canonical inclusion
$I^{\ov{p}}  C_{*}(U)\hookrightarrow C_{*}^{\ov p}(U)$.

(i) The Mayer-Vietoris exact sequences have been constructed in 
\thmref{thm:MVnewdim}
for the  complex $C^{\ov{p}}_{*}(X)$
and in 
\cite[Proposition 4.1]{CST3} (or
\cite[Theorem 4.4.19]{FriedmanBook})  for the complex $I^{\ov{p}}C_{*}(X)$.

(ii) This a classical argument for  homology theories.

(iii) Let $L$ be a compact filtered space such that the natural inclusion induces the isomorphism
$$\Phi_{(\R^i\times (\rc L\backslash \{\tv\})}\colon
I^{\ov{p}} H_{*}(\R^i\times (\rc L\backslash\{\tv\}))
\xrightarrow[]{\cong} 
H_{*}^{\ov p}(\R^i\times (\rc L\backslash\{\tv\})).$$
Since  $\R^i\times ]0,1[\times L =  \R^i\times (\rc L\backslash \{\tv\})$, we get the isomorphism
$$\Phi_{\R^i\times ]0,1[\times L}\colon 
I^{\ov{p}} H_{*}(\R^i\times ]0,1[ \times L)
\xrightarrow[]{\cong} 
H_{*}^{\ov p}(\R^i\times ]0,1[ \times L).$$
Let us consider the following commutative diagram
$$\xymatrix{
I^{\ov{p}} H_{*}(\R^i\times ]0,1[ \times L)
\ar[d]_{\pr_{*}}\ar[rr]^-{\Phi_{\R^i\times ]0,1[\times L}}
&&
H_{*}^{\ov p}(\R^i\times ]0,1[ \times L)
\ar[d]^{\pr_{*}}\\
I^{\ov{p}}H_{*}(L)
\ar[d]_{(\iota_{\rc L})_{*}}\ar[rr]^-{\Phi_{L}}
&&
 H_{*}^{\ov{p}}(L)
\ar[d]^{(\iota_{\rc L})_{*}}\\
I^{\ov{p}}H_{*}(\rc L)
\ar[rr]^-{\Phi_{\rc L}}
&&
H_{*}^{\ov{p}} (\rc L).
}$$
From \corref{cor:RfoisX} and \cite[Corollary 3.14]{CST3} (or \cite[Example 4.1.13.]{FriedmanBook}), we know that
 the two maps $\pr_{*}$, induced by the canonical projections, are isomorphisms. 
 We conclude that  $\Phi_{L}$ is an isomorphism.

If $*< n-\ov{p}( \{ \tv \})$ then \propref{prop:conenewdim} and 
\cite[Proposition 5.2]{CST3} (or \cite[Theorem 4.2.1]{FriedmanBook}) imply that the two maps 
$(\iota_{\rc L})_*$ are isomorphisms. So, $\Phi_{\rc L}$ is an isomorphism in these degrees.

When $*\geq n-\ov{p}( \{ \tv \})$, the map $\Phi_{\rc L}$ is directly an isomorphism (cf. \propref{prop:conenewdim} and \cite[Section 5.4]{FriedmanBook}).

(iv) The map $\Phi\colon H_{*}^{\ov{p}}(U)\to  I ^{\ov{p}}H_{*}(U)$ is the identity $G \to G$.
\hfill$\square$
%%%%%%%%%%%%%%
\section*{Acknowledgment}
We are grateful to the referee for her/his comments and suggestions which
helped us to improve the manuscript.

%%%%%%%%%%%%%%%%
\providecommand{\bysame}{\leavevmode\hbox to3em{\hrulefill}\thinspace}
\providecommand{\MR}{\relax\ifhmode\unskip\space\fi MR }
\providecommand{\MRhref}[2]{%
  \href{http://www.ams.org/mathscinet-getitem?mr=#1}{#2}
}
\providecommand{\href}[2]{#2}

\end{document}